\documentclass[11pt]{amsart}
\usepackage{amsmath}%
\usepackage{amsfonts}%
\usepackage{amssymb}%
\usepackage{graphicx, color, enumerate}
\usepackage{latexsym}
\usepackage{amscd}
\usepackage{mathrsfs}
\usepackage{cancel}
\usepackage{mathtools}
\usepackage{xcolor}
\usepackage[top=1in, bottom=1.25in, left=1in, right=1in]{geometry}
\usepackage[pdfborder=000,pdftex=true]{hyperref}
\newtheorem{theorem}{Theorem}[section]
\theoremstyle{plain}
\newtheorem*{theorem*}{Theorem}
\theoremstyle{plain}

\newtheorem{corollary}{Corollary}[section]
\newtheorem{definition}{Definition}[section]
\newtheorem{lemma}{Lemma}[section]

\newtheorem{proposition}{Proposition}[section]
\newtheorem{remark}{Remark}[section]

\numberwithin{equation}{section}
\DeclareRobustCommand{\rchi}{{\mathpalette\irchi\relax}}
\newcommand{\irchi}[2]{\raisebox{\depth}{$#1\chi$}}

\title[Concave-Convex Critical Fractional Problems with mixed boundary data]{Concave-Convex Critical problems for the Spectral Fractional Laplacian with Mixed Boundary Conditions}

\author{Alejandro Ortega}
\email[A. Ortega ]{alortega@math.uc3m.es}
\address[A. Ortega]{Departamento de Matem\'aticas,  
Universidad Carlos III de Madrid, Av. Universidad 30, 28911 Legan\'es (Madrid), Spain}

\thanks{This work has been supported by the Madrid Government (Comunidad de Madrid-Spain) under the Multiannual Agreement with UC3M in the line of Excellence of University Professors (EPUC3M23), and in the context of the V PRICIT (Regional Programme of Research and Technological Innovation).\\ The author is partially supported by the State Research Agency of Spain, under research project PID2019-106122GB-I00.}

\begin{document}
\maketitle
\begin{abstract}
In this work we study the existence of solutions to the following critical fractional problem with concave-convex nonlinearities, 
\begin{equation*}
        \left\{
        \begin{tabular}{lcl}
        $(-\Delta)^su=\lambda u^q+u^{2_s^*-1}$, &  $u>0$ &in $\Omega$, \\[2pt]
        $\mkern+51mu u=0$& &on $\Sigma_{\mathcal{D}}$, \\[2pt]
        $\mkern+36mu \displaystyle \frac{\partial u}{\partial \nu}=0$& &on $\Sigma_{\mathcal{N}}$,
        \end{tabular}
        \right.
\end{equation*}
where  $\Omega\subset\mathbb{R}^N$ is a smooth bounded domain,
$\frac{1}{2}<s<1$, $0<q<2_s^*-1$, $q\neq 1$, being $2_s^*=\frac{2N}{N-2s}$ the critical fractional Sobolev exponent,
$\lambda>0$, $\nu$ is the outwards normal to $\partial\Omega$,  $\Sigma_{\mathcal{D}}$,
$\Sigma_{\mathcal{N}}$ are smooth $(N-1)$-dimensional submanifolds
of $\partial\Omega$ such that
$\Sigma_{\mathcal{D}}\cup\Sigma_{\mathcal{N}}=\partial\Omega$,
$\Sigma_{\mathcal{D}}\cap\Sigma_{\mathcal{N}}=\emptyset$, and
$\Sigma_{\mathcal{D}}\cap\overline{\Sigma}_{\mathcal{N}}=\Gamma$ is
a smooth $(N-2)$-dimensional submanifold of $\partial\Omega$.\newline
In particular, we will prove that, for the sublinear case $0<q<1$, there exists at least two solutions for every $0<\lambda<\Lambda$ for certain $\Lambda\in\mathbb{R}$ while, for the superlinear case $1<q<2_s^*-1$,  we will prove that there exists at least one solution for every $\lambda>0$. We will also prove that solutions are bounded.
\end{abstract}

\section{Introduction}
In this work we study the existence of solutions to the following concave-convex critical problem involving the spectral fractional Laplacian,
\begin{equation}\label{p_lambda}\tag{$P_\lambda$}
        \left\{
        \begin{tabular}{lcl}
        $(-\Delta)^su=\lambda u^q+u^{2_s^*-1}$, & $u>0$ &in $\Omega$, \\
        $\mkern+21muB(u)=0$& &on $\partial\Omega$,
        \end{tabular}
        \right.
\end{equation}
where $\frac{1}{2}<s<1$, $0<q<2_s^*-1$, $q\neq 1$, being $2_s^*=\frac{2N}{N-2s}$ the critical fractional Sobolev exponent, $\lambda>0$, $\Omega\subset\mathbb{R}^N$ is a smooth bounded domain with mixed boundary conditions
\begin{equation}\label{mixed}
B(u)=u\rchi_{\Sigma_{\mathcal{D}}}+\frac{\partial u}{\partial
\nu}\rchi_{\Sigma_{\mathcal{N}}},
\end{equation}
where $\nu$ is the outwards normal to $\partial\Omega$, $\chi_A$ stands for the characteristic function of a set $A$, $\Sigma_{\mathcal{D}}$ and
$\Sigma_{\mathcal{N}}$ are smooth $(N-1)$-dimensional submanifolds of $\partial\Omega$ such that
$\Sigma_{\mathcal{D}}$ is a closed submanifold of $\partial\Omega$ with measure $|\Sigma_{\mathcal{D}}|=\alpha>0,\ \alpha\in(0,|\partial\Omega|)$; $\Sigma_{\mathcal{D}}\cap\Sigma_{\mathcal{N}}=\emptyset$,
$\Sigma_{\mathcal{D}}\cup\Sigma_{\mathcal{N}}=\partial\Omega$ and $\Sigma_{\mathcal{D}}\cap\overline{\Sigma}_{\mathcal{N}}=\Gamma$ is a smooth
$(N-2)$-dimensional submanifold. The range $\frac 12<s<1$ is the appropriate range for mixed boundary problems due to the natural embedding of the associated functional space, see Remark \ref{rem:rango_s}.\newline

Concave-convex critical problems are nowadays a well-known topic in the field of nonlinear PDE's as they have been broadly studied since the works of Brezis and Nirenberg (cf. \cite{Brezis1983a}) and Ambrosetti, Brezis and Cerami (cf. \cite{Ambrosetti1994}). The seminal paper \cite{Brezis1983a} deals with critical elliptic problems with Dirichlet boundary data for the classical Laplace operator $(s=1)$ and the exponent $q=1$. The authors introduced the main ideas to prove existence of solutions to critical problems with lower order perturbation terms. In \cite{Ambrosetti1994}, the authors analyzed the main effects concave-convex nonlinearities, $f_{\lambda}(u)=\lambda u^q+u^p$, $0<q<1<p\leq 2^*=\frac{N+2}{N-2}$, have on issues related to the
existence and multiplicity of solutions. Problems similar to \eqref{p_lambda} have been also studied for the $p$-Laplace operator (cf. \cite{Azorero1991}) or fully nonlinear operators (cf. \cite{Charro2009}) both considering Dirichlet boundary data. 

Results in these lines also hold when one considers the classical Laplace operator endowed now with mixed Dirichlet-Neumann boundary data. Lions, Pacella and Tricarico (cf. \cite{Lions1988}) analyzed the pure critical power problem $(\lambda=0)$ and the attainability of the associated Sobolev constant. In \cite{Grossi1990}, the corresponding mixed Brezis-Nirenberg problem was studied. Mixed concave-convex problems were addressed in \cite{Colorado2003} by considering subcritical nonlinearities and in \cite{Abdellaoui2006} by considering critical problems involving Caffarelli-Khon-Nirenberg weights. 

Regarding the fractional setting, the aim of this work is to extend to the mixed boundary data framework the results of \cite{Barrios2012}, where the Dirichlet problem ($|\Sigma_{\mathcal{N}}|=0$) corresponding to \eqref{p_lambda} was studied. Concave-convex critical fractional problems dealing with a different fractional operator, defined through a singular integral were studied in \cite{Barrios2015}. Fractional problems or nonlocal problems involving more general kernels and critical nonlinearities were studied in \cite{MolicaBisci2015, Servadei2013, Servadei2013a, Servadei2014, Servadei2015, Servadei2016}. %The mixed fractional Brezis-Nirenberg, corresponding to \eqref{p_lambda} with $q=1$, was addressed in \cite{Colorado2019}.

Coming back to problem \eqref{p_lambda}, using a generalized Pohozaev identity (cf. \cite{Pohozaev1965}) it can be seen that, under Dirichlet boundary conditions ($|\Sigma_{\mathcal{N}}|=0$), problem \eqref{p_lambda} has no solution for $\lambda=0$ and $\Omega$ a star-shaped domain (cf. \cite{Braendle2013}). Similar non existence results based on Pohozaev type identities for mixed problems can be found in \cite{Lions1988} and \cite{Colorado2019}. Nevertheless, mixed boundary critical problems behave quite differently from critical Dirichlet problems and, taking the Dirichlet boundary part small enough, one can prove the existence of positive solution for the pure critical problem corresponding to $\lambda=0$, (cf. \cite{Lions1988}, \cite[Theorem 2.1]{Abdellaoui2006} and  \cite[Theorem 2.9]{Colorado2019}). %It is worth to mention \cite{Bahri1985} where an existence result for domains $\Omega$ with non-trivial topology is established for pure critical problems involving the classical Laplace operator.
Indeed, \eqref{p_lambda} with $\lambda\geq0$ and the exponent $q=1$ was analyzed in \cite{Colorado2019} where it was proved the following. 
\begin{theorem}\cite[Theorem 1.1]{Colorado2019}\label{fractBN}
Assume that  $q=1$, $\frac 12<s<1$ and $N\geq 4s$. Let $\lambda_{1,s}$ be
the first eigenvalue of the fractional operator $(-\Delta)^s$ with
mixed Dirichlet-Neumann boundary conditions \eqref{mixed}. Then, the
problem \eqref{p_lambda}
\begin{enumerate}
\item has no solution for $\lambda\geq\lambda_{1,s}$,
\item has at least one solution for $0<\lambda<\lambda_{1,s}$,
\item has at least one solution for $\lambda=0$ and $|\Sigma_{\mathcal{D}}|$ small enough.
\end{enumerate}
\end{theorem}
Our aim is then to obtain existence results for problem \eqref{p_lambda} for the whole range of exponents $0<q<2_s^*-1,\, q\neq1$. Precisely we will prove the following two main results.

\begin{theorem}\label{sublinear}
Let $0<q<1$ and $\frac 12<s<1$. Then, there exists $0<\Lambda<\infty$ such that the problem \eqref{p_lambda}

\begin{enumerate}
\item has no solution for $\lambda>\Lambda$, 
\item has a minimal solution for any $0<\lambda<\Lambda$. Moreover, the family of minimal solutions is increasing with respect to $\lambda$,
\item has at least one solution for $\lambda=\Lambda$,
\item has at least two solutions for $0<\lambda<\Lambda$.
\end{enumerate}
\end{theorem}
\begin{theorem}\label{superlinear}
Let $1<q<2_s^*-1$, $\frac{1}{2}<s<1$ and $N>2s\left(1+\frac{1}{q}\right)$. Then, the problem \eqref{p_lambda} has at least one solution for any $\lambda>0$.

\end{theorem}

The proof of Theorem \ref{sublinear} follows from nowadays well-known arguments. The existence of a positive minimal solution follows by using sub and supersolution, comparison and iterative arguments. To prove the existence of a second positive solution we will need to use a recently proved Strong Maximum Principle for mixed fractional problems (cf. \cite{Ortega2021}), from which we will obtain a separation result (see Lemma \ref{lem:separation} below) that implies that the minimal solution is indeed a minimum of the energy functional associated to \eqref{p_lambda}. This step is fundamental to prove \textit{(4)} in Theorem \ref{sublinear} since it allows us to use a Mountain Pass type argument. Due to the lack of compactness of the Sobolev embedding at the critical exponent $2_s^*$, we prove next that a local PS condition holds below a certain critical level $c_{\mathcal{D}-\mathcal{N}}^*$. We conclude by constructing paths whose energy is below the critical level $c_{\mathcal{D}-\mathcal{N}}^*$. At this point we have two options as the mixed pure critical problem can have solution for Dirichlet boundary size small enough (Theorem \ref{fractBN} - \textit{(3)} above). If the Sobolev constant associated to \eqref{p_lambda} (see Definition \ref{defi_sob_const} below) is attained we use the associated extremal functions to find paths with energy below the critical level $c_{\mathcal{D}-\mathcal{N}}^*$. Otherwise, this step is accomplished by the use of appropriate truncations of the extremal functions of the fractional Sobolev inequality.\newline
Most of the arguments of the concave case $0<q<1$ also works for the convex case $q>1$ so we will only indicate the main steps to prove Theorem \ref{superlinear}.

\textbf{Organization of the paper:} In Section \ref{Functionalsetting} we introduce the appropriate functional setting and some results for a Sobolev-like constant associated to \eqref{p_lambda} useful in the sequel. In Section \ref{Section_concave} we will address the proof of Theorem \ref{sublinear}. We finish with the proof of Theorem \ref{superlinear} in Section \ref{Section_convex}.
%%%%%%%%%%%%%%%%%%%%%%%%%%%%%%%%%%%%
%%%%%%%%%%%%%%%%%%%%%%%%%%%%%%%%%%%%

\section{Functional setting and preliminaries}\label{Functionalsetting}

As far as the fractional Laplace operator is concerned, we recall its definition given through the spectral decomposition. Let
$(\varphi_i,\lambda_i)$ be the eigenfunctions (normalized with respect to the $L^2(\Omega)$-norm) and the eigenvalues of
$(-\Delta)$ equipped with homogeneous mixed Dirichlet--Neumann boundary data, respectively. Then, $(\varphi_i,\lambda_i^s)$ are the eigenfunctions and eigenvalues of the fractional operator
$(-\Delta)^s$, where, given $\displaystyle u_i(x)=\sum_{j\geq1}\langle u_i,\varphi_j\rangle\varphi_j$, $i=1,2$, it holds
\begin{equation*}
\langle(-\Delta)^s u_1, u_2\rangle=\sum_{j\ge 1} \lambda_j^s\langle u_1,\varphi_j\rangle \langle u_2,\varphi_j\rangle,
\end{equation*}
i.e., the action of the fractional operator on a smooth function $u$ is given by
\begin{equation*}
(-\Delta)^su=\sum_{j\ge 1} \lambda_j^s\langle u,\varphi_j\rangle\varphi_j.
\end{equation*}
As a consequence, the fractional Laplace operator $(-\Delta)^s$ is well defined through its spectral decomposition in the
following space of functions that vanish on $\Sigma_{\mathcal{D}}$,
\begin{equation*}
H_{\Sigma_{\mathcal{D}}}^s(\Omega)=\left\{u=\sum_{j\ge 1} a_j\varphi_j\in L^2(\Omega):\ u=0\ \text{on }\Sigma_{\mathcal{D}},\ \|u\|_{H_{\Sigma_{\mathcal{D}}^s}(\Omega)}^2=
\sum_{j\ge 1} a_j^2\lambda_j^s<\infty\right\}.
\end{equation*}
For $u\in H_{\Sigma_{\mathcal{D}}}^s(\Omega)$, it follows that $\displaystyle \|u\|_{H_{\Sigma_{\mathcal{D}}}^s(\Omega)}=\left\|(-\Delta)^{\frac{s}{2}}u\right\|_{L^2(\Omega)}$.
\begin{remark}\label{rem:rango_s}
As it is proved in \cite[Theorem 11.1]{Lions1972}, if $0<s\le \frac{1}{2}$ then $H_0^s(\Omega)=H^s(\Omega)$ and, therefore, also
$H_{\Sigma_{\mathcal{D}}}^s(\Omega)=H^s(\Omega)$, while for $\frac 12<s<1$, $H_0^s(\Omega)\subsetneq H^s(\Omega)$. Hence,
the range $\frac 12<s<1$ guarantees that $H_{\Sigma_{\mathcal{D}}}^s(\Omega)\subsetneq H^s(\Omega)$ and it provides us with the appropriate functional space for the mixed boundary problem \eqref{p_lambda}.
\end{remark}

This definition of the fractional powers of the Laplace operator allows us to integrate by parts in the proper spaces, so that a natural definition of weak solution to problem \eqref{p_lambda} is the following.
\begin{definition}
We say that $u\in H_{\Sigma_{\mathcal{D}}}^s(\Omega)$ is a weak solution to \eqref{p_lambda} if
\begin{equation}\label{energy_sol}
\int_{\Omega}(-\Delta)^{s/2}u(-\Delta)^{s/2}\psi dx=\int_{\Omega}\left(\lambda u^q+u^{2_s^*-1}\right)\psi dx,\ \ \text{for all}\ \psi\in H_{\Sigma_{\mathcal{D}}}^s(\Omega).
\end{equation}
\end{definition}
The right-hand side of \eqref{energy_sol} is well defined because of the embedding $H_{\Sigma_{\mathcal{D}}}^s(\Omega)\hookrightarrow L^{2_s^*}(\Omega)$, so $u\in H_{\Sigma_{\mathcal{D}}}^s(\Omega)$ then $\lambda u^q+u^{2_s^*-1}\in L^{\frac{2N}{N+2s}}\hookrightarrow \left(H_{\Sigma_{\mathcal{D}}}^s(\Omega)\right)'$.\newline
The energy functional associated to problem \eqref{p_lambda} is
\begin{equation}\label{energy_functional}
I_\lambda(u)=\frac{1}{2}\int_{\Omega}|(-\Delta)^{s/2}u|^2dx-\frac{\lambda}{q+1}\int_{\Omega}|u|^{q+1}dx-\frac{N-2s}{2N}\int_{\Omega}|u|^{\frac{2N}{N-2s}}dx.
\end{equation}
$I_\lambda$ is well defined in
$H_{\Sigma_{\mathcal{D}}}^s(\Omega)$ and positive critical points of $I_\lambda$ correspond to solutions of
\eqref{p_lambda}.\newline

Due to the nonlocal nature of the fractional operator $(-\Delta)^s$ some difficulties arise when one tries to obtain an explicit
expression of the action of the fractional Laplacian on a given function. In order to overcome these difficulties, we use the ideas by Caffarelli and Silvestre (see \cite{Caffarelli2007}) together with those of \cite{Braendle2013,Cabre2010,Capella2011}  to give an equivalent definition of the operator $(-\Delta)^s$ by means of an auxiliary problem that we introduce next.

Given a domain $\Omega \subset \mathbb{R}^N$, we set the cylinder $\mathscr{C}_{\Omega}=\Omega\times(0,\infty)\subset\mathbb{R}_+^{N+1}$. We denote by $(x,y)$ those points that belong to $\mathscr{C}_{\Omega}$ and by $\partial_L\mathscr{C}_{\Omega}=\partial\Omega\times[0,\infty)$ the lateral boundary of the cylinder. Let us also denote by  $\Sigma_{\mathcal{D}}^*=\Sigma_{\mathcal{D}}\times[0,\infty)$ and $\Sigma_{\mathcal{N}}^*=\Sigma_{\mathcal{N}}\times[0,\infty)$ as well as $\Gamma^*=\Gamma\times[0,\infty)$. It is clear that, by construction,
\begin{equation*}
\Sigma_{\mathcal{D}}^*\cap\Sigma_{\mathcal{N}}^*=\emptyset\,, \quad \Sigma_{\mathcal{D}}^*\cup\Sigma_{\mathcal{N}}^*=\partial_L\mathscr{C}_{\Omega} \quad \mbox{and} \quad \Sigma_{\mathcal{D}}^*\cap\overline{\Sigma_{\mathcal{N}}^*}=\Gamma^*\,.
\end{equation*}
 Then, given a function $u\in H_{\Sigma_{\mathcal{D}}}^s(\Omega)$ we define its $s$-harmonic extension, $w (x,y)=E_{s}[u(x)]$, as the solution to the problem
\begin{equation*}
        \left\{
        \begin{array}{rlcl}
           -\text{div}(y^{1-2s}\nabla w )&\!\!\!\!=0  & & \mbox{ in } \mathscr{C}_{\Omega} , \\
          B(w )&\!\!\!\!=0   & & \mbox{ on } \partial_L\mathscr{C}_{\Omega} , \\
          w(x,0)&\!\!\!\!=u(x)  & &  \mbox{ on } \Omega\times\{y=0\} ,
        \end{array}
        \right.
\end{equation*}
where
\begin{equation*}
B(w)=w\rchi_{\Sigma_{\mathcal{D}}^*}+\frac{\partial w}{\partial \nu}\rchi_{\Sigma_{\mathcal{N}}^* },
\end{equation*}
being $\nu$, with an abuse of notation\footnote{Let $\nu$ be the outward normal to $\partial\Omega$ and $\nu_{(x,y)}$ the outward normal to $\mathscr{C}_{\Omega}$ then $\nu_{(x,y)}=(\nu,0)$, $y>0$.}, the outward normal to $\partial_L\mathscr{C}_{\Omega}$. The extension function belongs to the space
\begin{equation*}
\mathcal{X}_{\Sigma_{\mathcal{D}}}^s(\mathscr{C}_{\Omega}) : =\overline{\mathcal{C}_{0}^{\infty}
((\Omega\cup\Sigma_{\mathcal{N}})\times[0,\infty))}^{\|\cdot\|_{\mathcal{X}_{\Sigma_{\mathcal{D}}}^s(\mathscr{C}_{\Omega})}},
\end{equation*}
where we define
\begin{equation}\label{norma}
\|\cdot\|_{\mathcal{X}_{\Sigma_{\mathcal{D}}}^s(\mathscr{C}_{\Omega})}^2:=\kappa_s\int_{\mathscr{C}_{\Omega}}\mkern-5mu y^{1-2s} |\nabla (\cdot)|^2dxdy,
\end{equation}
with $\kappa_s=2^{2s-1}\frac{\Gamma(s)}{\Gamma(1-s)}$ being $\Gamma(s)$ the Gamma function.

The space $\mathcal{X}_{\Sigma_{\mathcal{D}}}^s(\mathscr{C}_{\Omega})$ is a Hilbert space equipped with the norm
$\|\cdot\|_{\mathcal{X}_{\Sigma_{\mathcal{D}}}^s(\mathscr{C}_{\Omega})}$ which is induced by the scalar product
\begin{equation*}
\langle w, z \rangle_{\mathcal{X}_{\Sigma_{\mathcal{D}}}^s(\mathscr{C}_{\Omega})}=\kappa_s
\int_{\mathscr{C}_{\Omega}}y^{1-2s} \langle\nabla w,\nabla z\rangle dxdy.
\end{equation*}
Moreover, the following inclusions are satisfied,
\begin{equation} \label{embedd}
\mathcal{X}_0^s(\mathscr{C}_{\Omega}) \subsetneq \mathcal{X}_{\Sigma_{\mathcal{D}}}^s(\mathscr{C}_{\Omega}) \subsetneq \mathcal{X}^s(\mathscr{C}_{\Omega}),
\end{equation}
being  $\mathcal{X}_0^s(\mathscr{C}_{\Omega})$ the space of functions that belongs to $\mathcal{X}^s(\mathscr{C}_{\Omega})\equiv H^1(\mathscr{C}_{\Omega},y^{1-2s}dxdy)$ and vanish on the lateral boundary of $\mathscr{C}_{\Omega}$, denoted by $\partial_L\mathscr{C}_{\Omega}$.

The key point of the extension function is that it is related to the fractional Laplacian of the original function through the formula
\begin{equation*}
\frac{\partial w}{\partial \nu^s}:= -\kappa_s \lim_{y\to 0^+} y^{1-2s}\frac{\partial w}{\partial y}=(-\Delta)^su(x).
\end{equation*}
%In the case $\Omega=\mathbb{R}^N$ this formulation provides explicit expressions for both the fractional Laplacian and the $s$-extension in terms of the Riesz and the Poisson kernels, respectively. Namely,
%\begin{equation*}
%\begin{split}
%w(x,y)&=P_y^s\ast u(x)=c_{N,s}y^{2s}\int_{\mathbb{R}^N}\frac{u(z)}{(|x-z|^2+y^2)^{\frac{N+2s}{2}}}dz,\\
%(-\Delta)^{s}u(x)&=d_{N,s}P.V.\int_{\mathbb{R}^N}\frac{u(x)-u(y)}{|x-y|^{N+2s}},
%\end{split}
%\end{equation*}
%for certain constants $c_{N,s}$ and $d_{N,s}$ such that $2s\kappa_sc_{N,s}=d_{N,s}$, (cf. \cite{Braendle2013}).

By the above arguments, we can reformulate problem \eqref{p_lambda} in terms of the extension problem as follows
\begin{equation}\label{extension_problem}
        %P_{\lambda}\equiv
    \left\{
        \begin{array}{rlcl}
           -\text{div}(y^{1-2s}\nabla w )&\!\!\!\!=0  & & \mbox{ in } \mathscr{C}_{\Omega} , \\
          B(w)&\!\!\!\!=0   & & \mbox{ on } \partial_L\mathscr{C}_{\Omega} , \\
         \frac{\partial w}{\partial \nu^s}&\!\!\!\!=\lambda w^q+w^{2_s^*-1} & &  \mbox{ on } \Omega\times\{y=0\}.
        \end{array}
        \right.
        \tag{$P_{\lambda}^*$}
\end{equation}
A weak solution to \eqref{extension_problem} is a function $w\in \mathcal{X}_{\Sigma_{\mathcal{D}}}^s(\mathscr{C}_{\Omega})$ such that, for all $\varphi\in\mathcal{X}_{\Sigma_{\mathcal{D}}}^s(\mathscr{C}_{\Omega})$,
\begin{equation*}
\kappa_s\int_{\mathscr{C}_{\Omega}} y^{1-2s}\langle\nabla w,\nabla\varphi \rangle dxdy=\int_{\Omega} \left(\lambda w^q(x,0)+w^{2_s^*-1}(x,0)\right)\varphi(x,0)dx.
\end{equation*}
Given $w\in \mathcal{X}_{\Sigma_{\mathcal{D}}}^s(\mathscr{C}_{\Omega})$ a
solution to problem \eqref{extension_problem} the function
$u(x)=Tr[w](x)=w(x,0)$ belongs to
$H_{\Sigma_{\mathcal{D}}}^s(\Omega)$ and it is a solution to
problem \eqref{p_lambda} and vice versa, if $u\in
H_{\Sigma_{\mathcal{D}}}^s(\Omega)$ is a solution to \eqref{p_lambda}
then $w=E_s[u]\in
\mathcal{X}_{\Sigma_{\mathcal{D}}}^s(\mathscr{C}_{\Omega})$ is a
solution to \eqref{extension_problem}. Thus, both formulations are equivalent and the {\it extension operator}
$$
E_s: H_{\Sigma_{\mathcal{D}}}^s(\Omega) \to \mathcal{X}_{\Sigma_{\mathcal{D}}}^s(\mathscr{C}_{\Omega}),
$$
allows us to switch between each other. Moreover, according to \cite{Braendle2013, Caffarelli2007}, due to the choice of the constant $\kappa_s$, the extension operator $E_s$ is an isometry, i.e.,
\begin{equation}\label{norma2}
\|E_s[\varphi] \|_{\mathcal{X}_{\Sigma_{\mathcal{D}}}^s(\mathscr{C}_{\Omega})}=
\|\varphi \|_{H_{\Sigma_{\mathcal{D}}}^s(\Omega)}\quad \text{for all}\ \varphi\in H_{\Sigma_{\mathcal{D}}}^s(\Omega).
\end{equation}

Finally, the energy functional
associated to problem $(P_{\lambda}^*)$ is
\begin{equation}\label{extensionfunctional}
J_\lambda(w)=\frac{\kappa_s}{2}\int_{\mathscr{C}_{\Omega}}y^{1-2s}|\nabla w|^2dxdy-\frac{\lambda}{q+1}\int_{\Omega}|w|^{q+1}dx-\frac{N-2s}{2N}\int_{\Omega}|w|^{2_s^*}dx.
\end{equation}
Plainly, (positive) critical points of $J_\lambda$ in
$\mathcal{X}_{\Sigma_{\mathcal{D}}}^s(\mathscr{C}_{\Omega})$
correspond to (positive) critical points of $I_\lambda$ in
$H_{\Sigma_{\mathcal{D}}}^s(\Omega)$. Moreover, minima of $J_\lambda$ also
correspond to minima of $I_\lambda$. The proof of this fact is similar to
the one of the Dirichlet case, (cf. \cite[Proposition 3.1]{Barrios2012}).\newline

When one considers Dirichlet boundary conditions the following \textit{trace inequality} holds (cf. \cite[Theorem 4.4]{Braendle2013}): there exists $C=C(N,s,r,|\Omega|)>0$ such that, for all $z\in\mathcal{X}_0^s(\mathscr{C}_{\Omega})$,
\begin{equation}\label{sobext}
\kappa_s\int_{\mathscr{C}_{\Omega}}y^{1-2s}|\nabla z(x,y)|^2dxdy\geq C\left(\int_{\Omega}|z(x,0)|^rdx\right)^{\frac{2}{r}},
\end{equation}
for $1\leq r\leq 2_s^*,\ N>2s$. Because of \eqref{norma2} the trace inequality \eqref{sobext} is equivalent to the fractional Sobolev inequality,
\begin{equation}\label{sobolev}
  C\left(\int_{\Omega}|v|^rdx\right)^{\frac{2}{r}}\leq \int_{\Omega}|(-\Delta)^{\frac{s}2}v|^2dx
\qquad\text{for all } v\in H_{0}^s(\Omega),\ 1\leq r\leq 2^*_s ,\ N>2s.
\end{equation}
If $r=2_s^*$ the best constant in \eqref{sobolev} (and, thanks to \eqref{norma2}, in \eqref{sobext}), namely the fractional Sobolev constant, denoted by $S(N,s)$, is independent of the domain $\Omega$ and its exact value is given by 
\begin{equation*}
S(N,s)=2^{2s}\pi^s\frac{\Gamma\left(\frac{N+2s}{2}\right)}{\Gamma\left(\frac{N-2s}{2}\right)}\left(\frac{\Gamma(\frac{N}{2})}{\Gamma(N)}\right)^{\frac{2s}{N}}.
\end{equation*}
Since it is not achieved in any bounded domain we have that
\begin{equation*}
\kappa_s\int_{\mathbb{R}_{+}^{N+1}}\!\!y^{1-2s}|\nabla z(x,y)|^2dxdy\geq
S(N,s)\left(\int_{\mathbb{R}^N}|z(x,0)|^{\frac{2N}{N-2s}}dx\right)^{\frac{N-2s}{N}}
\!,\  \forall z\in \mathcal{X}^s(\mathbb{R}_{+}^{N+1}),
\end{equation*}
where
$\mathcal{X}^s(\mathbb{R}_{+}^{N+1})=\overline{\mathcal{C}^{\infty}(\mathbb{R}^N\times[0,\infty))}^{\|\cdot\|_{\mathcal{X}^s(\mathbb{R}_{+}^{N+1})}}$,
with $\|\cdot\|_{\mathcal{X}^s(\mathbb{R}_{+}^{N+1})}$ defined as
\eqref{norma} replacing $\mathscr{C}_\Omega$ by
$\mathbb{R}_{+}^{N+1}$. Indeed, in the whole space the latter inequality is achieved
for the family $w_{\varepsilon}= E_s[u_{\varepsilon}]$,
\begin{equation}\label{eq:sob_extremal}
u_{\varepsilon}(x)=\frac{\varepsilon^{\frac{N-2s}{2}}}{(\varepsilon^2+|x|^2)^{\frac{N-2s}{2}}},
\end{equation}
with arbitrary $\varepsilon>0$, (cf. \cite{Braendle2013}).

When mixed boundary conditions are considered the situation is quite similar since the Dirichlet condition is imposed on a set $\Sigma_{\mathcal{D}} \subset \partial \Omega$ such that $0<|\Sigma_{\mathcal{D}}|<|\partial\Omega|$. 

\begin{definition}\label{defi_sob_const}The Sobolev constant relative to the Dirichlet boundary $\Sigma_{\mathcal{D}}$ is defined by
\begin{equation*}
\widetilde{S}(\Sigma_{\mathcal{D}})=\inf_{\substack{u\in
H_{\Sigma_{\mathcal{D}}}^s(\Omega)\\ u\not\equiv
0}}\frac{\|u\|_{H_{\Sigma_{\mathcal{D}}}^s(\Omega)}^2}{\|u\|_{L^{2_s^*}(\Omega)}^2}.
\end{equation*}
Since the extension function minimizes the
$\|\cdot\|_{\mathcal{X}_{\Sigma_{\mathcal{D}}}^s(\mathscr{C}_{\Omega})}$
norm along all functions with the same trace on $\{y=0\}$ (cf. \cite[Lemma 2.4]{Colorado2019}), by \eqref{norma2}, the constant $\widetilde{S}(\Sigma_{\mathcal{D}})$ is equivalently defined by,
\begin{equation*}
\widetilde{S}(\Sigma_{\mathcal{D}})=\inf_{\substack{w\in \mathcal{X}_{\Sigma_{\mathcal{D}}}^s(\mathscr{C}_{\Omega})\\
w\not\equiv
0}}\frac{\|w\|_{\mathcal{X}_{\Sigma_{\mathcal{D}}}^s(\mathscr{C}_{\Omega})}^2}{\|w(\cdot,0)\|_{L^{2_s^*}(\Omega)}^2}.
\end{equation*}
\end{definition}
As mentioned above, since the Dirichlet condition is imposed on a set $\Sigma_{\mathcal{D}} \subset \partial \Omega$ such that $0<|\Sigma_{\mathcal{D}}|<|\partial\Omega|$, by the inclusions \eqref{embedd}, we have 
\begin{equation}\label{const}
0<\widetilde{S}(\Sigma_{\mathcal{D}})\vcentcolon=\inf_{\substack{u\in H_{\Sigma_{\mathcal{D}}}^s(\Omega)\\ u\not\equiv 0}}\frac{\|u\|_{H_{\Sigma_{\mathcal{D}}}^s(\Omega)}^2}{
\|u\|_{L^{2_s^*}(\Omega)}^2}
<\inf_{\substack{u\in H_{0}^s(\Omega)\\ u\not\equiv 0}}\frac{\|u\|_{H_{0}^s(\Omega)}^2}{\|u\|_{L^{2_s^*}(\Omega)}^2}.
\end{equation}
As we will see the constant $\widetilde{S}(\Sigma_{\mathcal{D}})$ plays a main role in the existence issues of problem \eqref{p_lambda}, i.e., $\widetilde{S}(\Sigma_{\mathcal{D}})$ is to mixed problems what Sobolev constant $S(N,s)$ is to Dirichlet problems.
\begin{remark}\label{remark_att}
Due to the spectral definition of the fractional operator, using H\"older's inequality we deduce $\widetilde{S}(\Sigma_{\mathcal{D}})\leq|\Omega|^{\frac{2s}{N}}\lambda_1^s(\alpha)$, with $\lambda_1(\alpha)$ the first eigenvalue of the Laplace operator endowed with mixed boundary conditions on the sets  $\Sigma_{\mathcal{D}}=\Sigma_{\mathcal{D}}(\alpha)$ and $\Sigma_{\mathcal{N}}= \Sigma_{\mathcal{N}}(\alpha)$. Since $\lambda_1(\alpha)\to0$ as $\alpha\to0^+$, (cf. \cite[Lemma 4.3]{Colorado2003}), we have $\widetilde{S}(\Sigma_{\mathcal{D}})\to0$ as $\alpha\to0^+$.
\end{remark}
Gathering together \eqref{const} and \eqref{norma2} it follows that, for all $\varphi \in \mathcal{X}_{\Sigma_{\mathcal{D}}}^s(\mathscr{C}_{\Omega})$,
\begin{equation*}
\widetilde{S}(\Sigma_{\mathcal{D}})\left(\int_\Omega |\varphi(x,0)|^{2^*_s} dx\right)^{\frac{2}{2^*_s}}\leq\|\varphi(x,0)\|_{H_{\Sigma_{\mathcal{D}}}^s(\Omega)}^2=\|E_s[\varphi(x,0)]\|_{\mathcal{X}_{\Sigma_{\mathcal{D}}}^s(\mathscr{C}_{\Omega})}^2.
\end{equation*}
This Sobolev--type inequality provides us with a trace inequality adapted to the mixed boundary data framework.
\begin{lemma}\cite[Lemma 2.4]{Colorado2019}\label{lem:traceineq}
For all $\varphi \in \mathcal{X}_{\Sigma_{\mathcal{D}}}^s(\mathscr{C}_{\Omega})$, we have 
\begin{equation*}
\widetilde{S}(\Sigma_{\mathcal{D}})\left(\int_\Omega|\varphi(x,0)|^{2^*_s}  )dx\right)^{\frac{2}{2^*_s}}\leq\kappa_s\int_{\mathscr{C}_{\Omega}} y^{1-2s} |\nabla \varphi|^2 dxdy.
\end{equation*}
\end{lemma}

Let us collect some results for $\widetilde{S}(\Sigma_{\mathcal{D}})$ proven in \cite{Colorado2019} useful in the sequel.

\begin{proposition}\cite[Proposition 3.6]{Colorado2019}\label{prop_cota}
If $\Omega$ is a smooth bounded domain, then $$\widetilde{S}(\Sigma_{\mathcal{D}})\leq 2^{\frac{-2s}{N}}S(N,s).$$
\end{proposition}

\begin{theorem}\cite[Theorem 2.9]{Colorado2019}\label{th_att}
If $\widetilde{S}(\Sigma_{\mathcal{D}})<2^{\frac{-2s}{N}}S(N,s)$ then $\widetilde{S}(\Sigma_{\mathcal{D}})$ is attained.
\end{theorem}

This result makes the difference between Dirichlet problems and mixed Dirichlet-Neumann problems. By taking $|\Sigma_{\mathcal{N}}|=0$ (Dirichlet case) and $\lambda=0$ in \eqref{p_lambda}, we have
the pure critical power problem which has no
positive solution, i.e., the Sobolev constant $S(N,s)$ is not attained, under some geometrical assumptions on $\Omega$,
for instance, under star-shapeness assumptions (cf. \cite{Braendle2013, Pohozaev1965}). Analogous non-existence results based on a Pohozaev--type identity and star-shapeness like assumptions holds for mixed problems (cf. \cite{Lions1988,Colorado2019}). Nevertheless, in the mixed case, the corresponding Sobolev constant
$\widetilde{S}(\Sigma_{\mathcal{D}})$ can be achieved thanks to
Theorem \ref{th_att}. Moreover, taking in mind Remark \ref{remark_att} the hypothesis of Theorem
\ref{th_att} can be fulfilled by making the size of the Dirichlet boundary part small enough (cf. \cite[Theorem 1.1-(3)]{Colorado2019}).\newline

At last, we enunciate a concentration-compactness result adapted to our
fractional setting with mixed boundary conditions useful in the sequel. First,
we recall the concept of a tight sequence.
\begin{definition}
We say that a sequence $\{y^{1-2s}|\nabla
w_n|^2\}_{n\in\mathbb{N}}\subset L^1(\mathscr{C}_{\Omega})$ is tight
if for any $\eta>0$ there exists $\rho>0$ such that
\begin{equation*}
\int_{\{y>\rho\}}\int_{\Omega}y^{1-2s}|\nabla w_n|^2dxdy\leq\eta,\ \forall n\in\mathbb{N}.
\end{equation*}
\end{definition}
%Note that the tight condition avoid the evanescence.

\begin{theorem}\cite[Theorem 4.4]{Colorado2019}\label{concompact}
Let $\{w_n\}\subset\mathcal{X}_{\Sigma_{\mathcal{D}}}^s(\mathscr{C}_{\Omega})$ be a
weakly convergent sequence to $w$ in $\mathcal{X}_{\Sigma_{\mathcal{D}}}^s(\mathscr{C}_{\Omega})$
such that $\{y^{1-2s}|\nabla w_n|^2\}_{n\in\mathbb{N}}$ is tight. Let us denote $u_n=Tr[w_n]$, $u=Tr[w]$ and let $\mu,\nu$ be two
nonnegative measures such that, in the sense of measures,

\begin{equation}\label{limitcompactness}
\kappa_sy^{1-2s}|\nabla w_n|^2\to\mu\quad \text{and}\quad |u_n|^{2_s^*}\to\nu,
\end{equation}
Then there exist an at most countable set $\mathfrak{I}$ and points $\{x_i\}_{i\in \mathfrak{I}}\subset\overline{\Omega}$ such that
\begin{enumerate}
\item $\nu=|u|^{2_s^*}+\sum\limits_{i\in \mathfrak{I}}\nu_{i}\delta_{x_{i}},\ \nu_{i}>0,$
\item $\mu\geq\kappa_sy^{1-2s}|\nabla w|^2+\sum\limits_{i\in \mathfrak{I}}\mu_{i}\delta_{x_{i}},\ \mu_{i}>0,$
\item $\mu_i\geq \widetilde{S}(\Sigma_{\mathcal{D}})\nu_i^{\frac{2}{2_s^*}}$.
\end{enumerate}
\end{theorem}

Using Theorem \ref{concompact} and Brezis-Lieb Lemma (cf. \cite{Brezis1983}) it can be proved the following.
\begin{theorem}\cite[Theorem 4.5]{Colorado2019}\label{CCA}
Let $w_m$ be a minimizing sequence of $\widetilde{S}(\Sigma_{\mathcal{D}})$. Then
either $w_m$ is relatively compact or the weak limit, $w\equiv
0$. Even more, in the latter case there exist a subsequence $w_{m}$
and a point $x_0\in\overline{\Sigma}_{\mathcal{N}}$ such that
\begin{equation}\label{acumulacion}
\kappa_sy^{1-2s}|\nabla w_m|^2\to \widetilde{S}(\Sigma_{\mathcal{D}})\delta_{x_0}\quad \text{and}\quad |u_m|^{2_s^*}\to\delta_{x_0},
\end{equation}
with $u_m=Tr[w_m]$.
\end{theorem}
%%%%%%%%%%%%%%%%%%%%%%%%%%%%%%%%%%%%
%%%%%%%%%%%%%%%%%%%%%%%%%%%%%%%%%%%%

\section{Concave case: $0<q<1$}\label{Section_concave}

In this section we prove Theorem \ref{sublinear}. The existence of a positive minimal solution and related results follow from nowadays well known arguments so we will be brief in details. The existence of a second positive solution follows from the following scheme: first we prove a separation result deduced from a recent work of the author (cf. \cite{Ortega2021}) which implies that the minimal solution is a minimum of the energy functional $I_\lambda$. Next, due to the lack of compactness of the Sobolev embedding at the critical exponent $2_s^*$, we prove that a local PS condition holds below the critical level $c_{\mathcal{D}-\mathcal{N}}^*=\frac{s}{N}[\widetilde{S}(\Sigma_{\mathcal{D}})]^{\frac{N}{2s}}$. Finally, we construct paths below the critical level $c_{\mathcal{D}-\mathcal{N}}^*$. This is done either using the extremal functions of the constant $\widetilde{S}(\Sigma_{\mathcal{D}})$ (in case this constant is attained) or, taking in mind Theorems \ref{limitcompactness} and \ref{CCA}, by concentrating the extremal functions of the Sobolev inequality at points on the Neumann boundary $\Sigma_{\mathcal{N}}$. 

Let us consider the following fractional mixed problem with a general nonlinearity,
\begin{equation}\label{prob:general_f}
        \left\{
        \begin{tabular}{lcl}
        $(-\Delta)^su=f(x,u),$ & $u>0$ &in $\Omega$, \\
        %$\mkern+50mu u>0$ &   &in $\Omega$, \\[2pt]
        $\mkern+22mu B(u)=0$& &on $\partial\Omega$.
        \end{tabular}
        \right.
        \tag{$P_{f}$}
\end{equation}
Note that \eqref{p_lambda} corresponds to \eqref{prob:general_f} for $f=f_\lambda(t)=\lambda|t|^{q-1}t+|t|^{2_s^*-1}t$. The associated energy functional, $I_f: H_{\Sigma_{\mathcal{D}}}^s(\Omega)\mapsto\mathbb{R}$, is given by 
\begin{equation*}
I_f(u)=\frac{1}{2}\int_{\Omega}|(-\Delta)^{s/2}u|^2dx-\int_{\Omega}F(w(x,0))dx,
\end{equation*}
where $F(t)=\int_0^\tau f(\tau)d\tau$ (resp. $F_\lambda(t)=\int_0^\tau f_\lambda(\tau)d\tau=\frac{1}{q+1}|t|^{q+1}+\frac{1}{2_s^*}|t|^{2_s^*}$). The corresponding extension problem reads
\begin{equation}\label{prob:general_f_extension}
        \left\{
        \begin{array}{rlcl}
        \displaystyle -\text{div}(y^{1-2s}\nabla w)&\!\!\!\!=0  & & \mbox{ in } \mathcal{C}_{\Omega} , \\
        \displaystyle B(w)&\!\!\!\!=0   & & \mbox{ on } \partial_L\mathcal{C}_{\Omega} , \\
        w&\!\!\!\!>0   & &\mbox{ on } \Omega\times\{y=0\},
        \\
         \displaystyle \frac{\partial w}{\partial \nu^s}&\!\!\!\!= f(w) & &  \mbox{ on } \Omega\times\{y=0\}.
         \end{array}
         \right.
         \tag{$P_{f}^*$}
\end{equation}
The associated energy functional, $J_f: \mathcal{X}_{\Sigma_{\mathcal{D}}}^s(\mathscr{C}_{\Omega})\mapsto\mathbb{R}$, is
\begin{equation*}
J_f(w)=\frac{\kappa_s}{2}\int_{\mathscr{C}_{\Omega}}y^{1-2s}|\nabla w|^2dxdy-\int_{\Omega}F(w(x,0))dx.
\end{equation*}

\begin{lemma}
A function $u_0\in H_{\Sigma_{\mathcal{D}}}^s(\Omega)$ is a local minimum of the energy functional $I_\lambda$ if and only if $w_0=E_s[u_0]\in \mathcal{X}_{\Sigma_{\mathcal{D}}}^s(\mathscr{C}_{\Omega})$ is a local minimum of the energy functional $J_{\lambda}$.
\end{lemma}
\begin{proof}
The follows exactly as that of \cite[Proposition 3.1]{Barrios2012}, so we omit the details.
\end{proof}

\begin{proposition}\label{prop:bound}
Let $u\in H_{\Sigma_{\mathcal{D}}}^s(\Omega)$ be a solution to \eqref{prob:general_f} with the nonlinearity $f$ satisfying 
\begin{equation*}
0\leq f(x,t)\leq C(1+|t|^p)\quad\text{for all } (x,t)\in\Omega\times\mathbb{R}\ \text{and } 0< p\leq 2_s^*-1.
\end{equation*}
Then $u\in L^{\infty}(\Omega)$.
\end{proposition}

\begin{proof}
Following verbatim the proof of \cite[Proposition 5.1]{Barrios2012}, by a truncation argument we deduce $u^{\beta+1}\in L^{2_s^*}(\Omega)$ for $\beta\geq0$. Then, by an iteration argument, we get $f(\cdot,u)\in L^r(\Omega)$, $r>\frac{N}{2s}$ after a finite number of steps. Because of \cite[Theorem 3.7]{Carmona2020} we conclude $u\in L^{\infty}(\Omega)$.
\end{proof}
The next result deals with the sub and supersolutions method.
\begin{lemma}\cite[Lemma 5.2]{Carmona2020a}\label{existencia}
Suppose that there exist a subsolution $w_1$ and a supersolution $w_2$ to \eqref{prob:general_f_extension}, i.e., $w_1,w_2 \in \mathcal{X}_{\Sigma_{\mathcal{D}}}^s(\mathscr{C}_{\Omega})$ such that, for
any nonnegative $\phi \in \mathcal{X}_{\Sigma_{\mathcal{D}}}^s(\mathscr{C}_{\Omega})$,
\begin{equation*}
\begin{split}
\kappa_s\int_{\mathscr{C}_{\Omega}}y^{1-2s}\nabla w_1\nabla \phi dxdy&\leq\int_{\Omega}f(w_1(x,0))\phi (x,0)dx\,, \\
\kappa_s\int_{\mathscr{C}_{\Omega}}y^{1-2s}\nabla w_2\nabla \phi dxdy&\geq\int_{\Omega}f(w_2(x,0))\phi (x,0) dx\,.
\end{split}
\end{equation*}
\vskip -2pt \noindent
Assume moreover that $w_1\leq w_2$ in $\mathscr{C}_{\Omega}$. Then, there  exists a solution $w$ to problem \eqref{prob:general_f_extension} verifying $w_1\leq w \leq w_2$ in $\mathscr{C}_{\Omega}$.
\end{lemma}
Finally we recall a comparison result.
\begin{lemma}\cite[Lemma 5.3]{Carmona2020a}\label{orden}
Let $w_1,w_2\in \mathcal{X}_{\Sigma_{\mathcal{D}}}^s(\mathscr{C}_{\Omega})$ be respectively a positive subsolution and a positive supersolution to \eqref{prob:general_f_extension} and assume that $f(t)/t$ is decreasing for $t>0$. Then $w_1\leq w_2$ in $\mathscr{C}_{\Omega}$.
\end{lemma}

We address now the proof of Theorem \ref{sublinear}.
\begin{lemma} \label{lem:Lambda}
Let $\Lambda$ be defined by
\begin{equation*}
\Lambda=\sup\{\lambda>0:\eqref{p_lambda}\  \text{has solution} \},
\end{equation*}
\vskip -4pt \noindent
then, $0<\Lambda<\infty$.
\end{lemma}

\vspace*{-4pt}

\begin{proof}
Let $(\lambda_1,\varphi_1)$ be the first eigenvalue and a corresponding positive eigenfunction of the fractional Laplacian in $\Omega$. Using $\varphi_1$ as a test function in
problem \eqref{p_lambda}, we have
\vskip -11pt
\begin{equation}\label{eigen}
\int_{\Omega}(\lambda
u^q+u^{2_s^*-1})\varphi_1dx=\lambda_1^s\int_{\Omega}u\varphi_1dx.
\end{equation}
Since there exists constants $c=c(N,s,q)<1$ and $\delta=\frac{2_s^*-2}{2_s^*-q}$ such that $\lambda t^q+t^{2_s^*-1}>c\lambda^{\delta}t$ for any $t>0$, from \eqref{eigen} we deduce
$c\lambda^{\delta}<\lambda_1^s$ and hence $\Lambda<\infty$. In
particular, this also proves that there is no solution to
\eqref{p_lambda} for $\lambda>\Lambda$.

In order to prove  that $\Lambda>0$, we prove, by means of the sub and supersolution technique, the existence of solution to \eqref{extension_problem} for any small positive $\lambda$. Indeed, for $\varepsilon>0$ small enough, $\underline{U}=\varepsilon E_s[\varphi_1]$ is a subsolution to \eqref{extension_problem}. Because of Proposition \ref{prop:bound}, a supersolution can be constructed as an appropiate multiple of the function $G$, the  solution to
\begin{equation*}
        \left\{
        \begin{array}{rlcl}
        \displaystyle -\text{div}(y^{1-2s}\nabla G)&\!\!\!\!=0  & & \mbox{ in } \mathcal{C}_{\Omega} , \\
        \displaystyle B(G)&\!\!\!\!=0   & & \mbox{ on } \partial_L\mathcal{C}_{\Omega} , \\
         \displaystyle \frac{\partial G}{\partial \nu^s}&\!\!\!\!= 1& &  \mbox{ on } \Omega\times\{y=0\}.
         \end{array}
         \right.
\end{equation*}
Note that, as the trace function $g(x)=G(x,0)$ is a solution to
\begin{equation*}
        \left\{
        \begin{tabular}{rcl}
        $(-\Delta)^sg=1$ & &in \ $\Omega$, \\[2pt]
        $B(g)=0$  & &on \ $\partial\Omega$,
        \end{tabular}
        \right.
\end{equation*}
by \cite[Theorem 3.7]{Carmona2020} we have $\|g\|_{L^{\infty}(\Omega)}<+\infty$. Next, since $0<q<1$ we can find $\lambda_0>0$ such that for all $0<\lambda\leq\lambda_0$ there exists $M=M(\lambda)$ such that
\begin{equation}\label{eq:Mlto0}
M\geq\lambda M^q\|g\|_{L^{\infty}(\Omega)}^q+M^{2_s^*-1}\|g\|_{L^{\infty}(\Omega)}^{2_s^*-1}.
\end{equation}
As a consequence, the function $h=Mg$ satisfies $M=(-\Delta)^sh\geq \lambda h^q+h^{2_s^*-1}$ and, by the Maximum Principle (cf. \cite[Lemma 2.3]{Capella2011}), the extension function
$\overline U=E_s[h]$ is a supersolution and $\underline U \leq \overline U$. Applying Lemma \ref{existencia} we conclude the existence of a
solution $U(x,y)$ to problem \eqref{extension_problem}. Therefore, its trace $u(x)=U(x,0)$ is a solution to problem \eqref{p_lambda} with $\lambda<\lambda_0$.
\end{proof} %%%%%%%%%%%%%%
\begin{lemma}\label{lem:minimal}
Problem \eqref{p_lambda} has at least a positive minimal solution for every $0<\lambda<\Lambda$. Moreover, the family $\{u_{\lambda}\}$ of minimal solutions is increasing with respect to $\lambda$.
\end{lemma}

\begin{proof}
By definition of $\Lambda$, for any $0<\lambda<\Lambda$ there exists $\mu\in (\lambda,\Lambda]$ such that $(P_{\mu}^*)$ has a solution $W_{\mu}$.
It is easy to see that  $W_{\mu}$ is a supersolution for \eqref{extension_problem}.

On the other hand, let $V_{\lambda}$ be the unique solution to problem \eqref{prob:general_f_extension} with $f(t)= \lambda t^q$ (the existence can be deduced by minimization, while uniqueness follows from Lemma~\ref{orden}). It is clear that $V_{\lambda}$ is a subsolution to problem \eqref{extension_problem} and, by Lemma \ref{orden}, we have $V_{\lambda} \leq W_{\mu}$. Thus, by Lemma \ref{existencia}, we conclude that there is a solution to \eqref{extension_problem} and, as a consequence, for the whole open interval $(0,\Lambda)$.

Finally, we prove the existence of a minimal solution for all $0<\lambda<\Lambda$. Given a solution $u$ to \eqref{p_lambda} we take $U=E_s[u]$ with $U$ solution to \eqref{extension_problem}. By Lemma~\ref{orden}, we have $V_{\lambda}\leq U$ with $V_{\lambda}$ solution to \eqref{prob:general_f_extension} with $f(t)= \lambda t^q$. Then, $v_{\lambda}(x)=V_{\lambda}(x,0)$ is a subsolution to \eqref{p_lambda}
and the monotone iteration
\begin{equation*}
\begin{array}{c}
(-\Delta)^s u_{n+1}=\lambda u_{n}^q+u_n^r,\quad  u_n \in H_{\Sigma_{\mathcal{D}}}^s ({\Omega})\quad \mbox{ with } \quad u_0=v_{\lambda},
 \end{array}
\end{equation*}
verifies $u_n\leq U(x,0)=u$ and $u_n\nearrow u_{\lambda}$ with $u_{\lambda}$ solution to problem \eqref{p_lambda}. In particular, $u_\lambda\leq u$ and we conclude that $u_\lambda$ is a minimal solution. The monotonicity follows directly from the first part of the proof, taking $U_\mu=E_s[u_\mu]$ which leads to $u_\lambda\leq u_\mu$ whenever $0<\lambda < \mu\leq \Lambda$.
\end{proof} %%%%%%%%%%%%%%%%%%%

\vspace*{-5pt}

\begin{lemma}\label{lem:solLambda}
Problem \eqref{p_lambda} has at least one solution if $\lambda = \Lambda$.
\end{lemma}

To prove Lemma \ref{lem:solLambda} we need the following result that guarantees that the linearized equation corresponding to \eqref{p_lambda} has non-negative eigenvalues at the minimal solution.

\vspace*{-3pt}

\begin{proposition}
Let $u_{\lambda}\in H_{\Sigma_{\mathcal{D}}}^s(\Omega)$ be the minimal solution to problem \eqref{p_lambda} and let us define $a_{\lambda}=a_{\lambda}(x)=\lambda q u_{\lambda}^{q-1}+(2_s^*-1) u_{\lambda}^{2_s^*-2}$. Then, the operator $[(-\Delta)^s-a_\lambda(x)]$ with mixed boundary conditions has a first eigenvalue $\nu_1\geq0$.
In particular, it follows that
\begin{equation}\label{ineq:posit}
\int_{\Omega}\left(|(-\Delta)^{s/2}v|^2-a_{\lambda}v^2\right)dx\geq0,\quad  \mbox{for any } \ \,  v\in H_{\Sigma_{\mathcal{D}}}^s(\Omega).
\end{equation}
\end{proposition}

\begin{proof}
The proof follows verbatim that of \cite[Proposition 5.1]{Carmona2020a}, so we omit the details.
\end{proof} %%%%%%%%%%%%%%%%%%%

\medskip %%

\begin{proof}[Proof of Lemma \ref{lem:solLambda}]
Let $\{\lambda_n\}$ be a sequence such that $\lambda_n \nearrow \Lambda$ and denote by $u_n=u_{\lambda_n}$ the minimal solution to problem $(P_{\lambda_n})$.
Let $U_n=E_s[u_n]$, then
\begin{equation*}
I_{\lambda_n}(u_n)=\frac{1}{2}\int_{{\Omega}}|(-\Delta)^{\frac{s}{2}}u_n|^2dx-\frac{\lambda_n}{q+1}\int_{\Omega} |u_n|^{q+1}dx-\frac{1}{2_s^*}\int_{\Omega}|u_n|^{2_s^*}dx.
\end{equation*}
Moreover, since $u_n$ is a solution to $(P_{\lambda_n})$, it also satisfies
\vskip -10pt
\begin{equation*}
\int_{{\Omega}}|(-\Delta)^{\frac{s}{2}}u_n|^2dx=\lambda_n\int_{\Omega} |u_n|^{q+1}dx+\int_{\Omega}|u_n|^{2_s^*}dx.
\end{equation*}
On the other hand, using \eqref{ineq:posit} with $v=u_n$,
\begin{equation*}
\int_{{\Omega}}|(-\Delta)^{\frac{s}{2}}u_n|^2dx-\lambda_nq\int_{\Omega} |u_n|^{q+1}dx-(2_s^*-1)\int_{\Omega}|u_n|^{2_s^*}dx\geq0.
\end{equation*}
As in \cite[Theorem 2.1]{Ambrosetti1994}, we  conclude $I_{\lambda_n}(u_n)<0$. Moreover, as $I_{\lambda_n}'(u_n) = 0$, then
${\|u_n\|}_{H_{\Sigma_{\mathcal{D}}}^s(\Omega)}\leq C$. Hence, there exists a weakly convergent subsequence $u_n\to u\in
H_{\Sigma_{\mathcal{D}}}^s(\Omega)$ and, thus, $u$ is a weak solution of \eqref{p_lambda} for $\lambda=\Lambda$.
\end{proof}
Having proved the first three items of Theorem \ref{sublinear}, next we focus on proving the existence of a second solution. As commented above, first we show that the minimal solution is a local minimum of the energy functional $I_\lambda$; so we can use the Mountain Pass Theorem, obtaining a minimax PS sequence. To find a second solution, we prove next a local PS$_c$ condition for energy levels $c$ under a critical level $c_{\mathcal{D-N}}^*$. Let us stress that, in order to obtain a second solution, it is fundamental that the minimal solution is a minimum of the energy functional $I_\lambda$ or, equivalently, its $s$-harmonic extension $w_\lambda=E_s[u_\lambda]$ is a minimum of $J_\lambda$.

Following the ideas of \cite{Colorado2003} we begin with a separation Lemma. Let us consider $v$ solution to 
\begin{equation}\label{p_v}%\tag{$P_v$}
        \left\{
        \begin{tabular}{rcl}
        $(-\Delta)^sv=g$ & &in  $\Omega$,\\[2pt]
        $B(v)=0\mkern+0.5mu$& &on  $\partial\Omega$,
        \end{tabular}
        \right.\quad\text{with } g\in L^p(\Omega),\ p>\frac{N}{s}.
\end{equation}
%{\color{red} For $g\equiv1$ the function $v$ plays a role in the mixed problems similar to that of the function $d(x)=dist(x,\partial\Omega)$ in the Dirichlet problems.}
The following result is proven in \cite{Ortega2021}.

\begin{theorem}\cite[Theorem 1.2]{Ortega2021}\label{th_stMax}
Let $u$ be the solution to 
\begin{equation*}
        \left\{
        \begin{tabular}{rcl}
        $(-\Delta)^su=f$ & &in $\Omega$, \\[2pt]
        $B(u)=0$  & &on $\partial\Omega$,
        \end{tabular}
        \right.
\end{equation*}
with $f\in L^\infty(\Omega)$, $f\gneq0$ and let $v$ be the solution to \eqref{p_v}. Then, there exists a constant $C>0$ such that 
\begin{equation*}
\left\| \ \frac vu   \  \right \|_ {L^\infty(\Omega)}\leq C\|g\|_{L^p(\Omega)},
\end{equation*}
with the constant $C$ depending on $N$, $p$, $s$, $\Omega$, $\Sigma_{\mathcal{D}}$, $\|u\|_{L^{\infty}(\Omega)}$, $\|f\|_{L^{\infty}(\Omega)}$ and $1/(\int_{\Omega}f(z)d(z)dz)$ where $d(x)=\textrm{dist}(x,\partial\Omega)$.
\end{theorem}

Let us define the class
\begin{equation*}
\mathfrak{C}_v(\Omega)=\left\{\omega\in\mathcal{C}^0(\overline{\Omega})\cap H_{\Sigma_{\mathcal{D}}}^s(\Omega): \left\|\frac{\omega}{v}\right\|_{L^{\infty}(\Omega)}<\infty\right\}.
\end{equation*}
By application of Theorem \ref{th_stMax} with $g\equiv1$ we obtain the following separation result in the $\mathfrak{C}_v(\Omega)$-topology. Then, from now  $v$ denotes the solution to \eqref{p_v} with $g\equiv1$.

\begin{lemma}\label{lem:separation}
Let $0<\lambda_0<\lambda_1<\lambda_2<\Lambda$ and $u_{\lambda_0}, u_{\lambda_1}$ and $u_{\lambda_2}$ be the minimal solutions to \eqref{p_lambda} with $\lambda=\lambda_0, \lambda_1$ and $\lambda_2$ respectively. Then, if $X=\{\omega\in \mathfrak{C}_v(\Omega): u_{\lambda_0}\leq\omega\leq u_{\lambda_2}\}$, there exists $\varepsilon>0$ such that 
\begin{equation*}
u_{\lambda_1}+\varepsilon B_1(0)\subset X,
\end{equation*}
where $B_1(0)=\{\omega\in \mathfrak{C}_v(\Omega):\left\|\frac{\omega}{v}\right\|_{L^{\infty}(\Omega)}<1\}$.
\end{lemma}
\begin{proof}
By the Maximum Principle (cf. \cite[Lemma 2.3]{Capella2011}), $u_{\lambda_0}<u_{\lambda_1}<u_{\lambda_2}$. Next, note that the function $\underline{v}\vcentcolon= u_{\lambda_1}-u_{\lambda_0}$ solves $(-\Delta)^s\underline{v}=f_{\lambda_1}(u_{\lambda_1})-f_{\lambda_0}(u_{\lambda_0})$ and the function $\overline{v}\vcentcolon= u_{\lambda_2}-u_{\lambda_1}$ solves $(-\Delta)^s\overline{v}=f_{\lambda_2}(u_{\lambda_2})-f_{\lambda_1}(u_{\lambda_1})$, in both cases with the same boundary condition as in problem \eqref{p_lambda}. Since $f_\lambda(t)$ is increasing the right-hand sides are nonnegative and, by Proposition \ref{prop:bound}, they are also bounded. Then, by Theorem \ref{th_stMax}, there exists $\varepsilon>0$ such that 
\begin{equation*}
u_{\lambda_0}(x)+\varepsilon v(x)\leq u_{\lambda_1}(x)\leq u_{\lambda_2}(x)-\varepsilon v(x),\quad\text{for all } x\in\overline{\Omega},
\end{equation*}
and the result follows.
\end{proof}
Let us define the functions 
\begin{equation*}
\overline{F}_\lambda(u)=\int_{0}^{u}\overline{f}_\lambda(t)dt\qquad\text{with}\qquad
\overline{f}_\lambda(t)=\left\{
        \begin{tabular}{lr}
         $\lambda t^q+t^{2_s^*-1}$&if $t\geq0$, \\[5pt]
         $0$&if $t<0$,
        \end{tabular}
        \right.   
\end{equation*}
and the energy functional 
\begin{equation*}
\overline{I}_\lambda(u)=\frac{1}{2}\|u\|_{H_{\Sigma_{\mathcal{D}}}^s(\Omega)}^2-\int_{\Omega}\overline{F}_\lambda(u)dx.
\end{equation*}
Clearly $\overline{I}_\lambda(u)=I_\lambda(u)$, if $u>0$. The Euler--Lagrange equation corresponding to $\overline{I}_\lambda(u)$ is 
\begin{equation}\label{p_lambda+}\tag{$P_\lambda^+$}
        \left\{
        \begin{tabular}{lcl}
        $(-\Delta)^su=\lambda (u^+)^q+(u^+)^{2_s^*-1}$ & &in  $\Omega$, \\
        $\mkern+22muB(u)=0$& &on  $\partial\Omega$,
        \end{tabular}
        \right.
\end{equation}
where $u^+=\max\{0,u\}$. If $u$ is a nontrivial solution to \eqref{p_lambda+}, i.e., $u$ is a nontrivial critical point of $\overline{I}_\lambda$ then, by the Maximum Principle (cf. \cite[Lemma 2.3]{Capella2011}), $u$ is strictly positive in $\Omega$ and, therefore, it is also a solution to \eqref{p_lambda}, i.e., a positive critical point of $I_\lambda$; and vice-versa, if $u$ is a solution to \eqref{p_lambda}, then $u>0$ in $\Omega$ and $u^+=u$, so that $u$ is a solution to \eqref{p_lambda+}. 
\begin{proposition}\label{prop:min_C}
For all $\lambda\in(0,\Lambda)$ there exists a solution $u_0\in H_{\Sigma_{\mathcal{D}}}^s(\Omega)$ to \eqref{p_lambda} which is a local minimum of the energy functional $\overline{I}_{\lambda}$ in $\mathfrak{C}_v(\Omega)$, i.e., there exists $r>0$ such that 
\begin{equation*}
\overline{I}_{\lambda}(u_0)\leq \overline{I}_{\lambda}(\varphi),\quad \text{for all } \omega\in \mathfrak{C}_v(\Omega)\text{ with } \left\|\frac{u_0-\omega}{v}\right\|_{L^{\infty}(\Omega)}\leq r.
\end{equation*}
\end{proposition}
\begin{proof} 
%The proof follows as in \cite[Lemma 5.5]{Colorado2003} which in turn is based on \cite[Lemma 4.1]{Ambrosetti1994}. 
Fixed $0<\lambda_1<\lambda<\lambda_2<\Lambda$ let us consider the corresponding minimal solutions $u_{\lambda_1}$ and $u_{\lambda_2}$. Note that $u_{\lambda_1}\leq u_{\lambda_2}$ and $u_{\lambda_1}$, $u_{\lambda_2}$ are a subsolution and a supersolution to \eqref{p_lambda}. Moreover, setting $z\vcentcolon= u_{\lambda_2}-u_{\lambda_1}$, then 
\begin{equation*}
        \left\{
        \begin{tabular}{lcl}
        $(-\Delta)^sz\geq0$ & &in  $\Omega$, \\
        $\mkern+22muB(z)=0$& &on  $\partial\Omega$.
        \end{tabular}
        \right.
\end{equation*}
Next, let us take the functions
\begin{equation*}
F_\lambda^*(u)=\int_{0}^{u}f_\lambda^*(x,t)dt\qquad\text{with}\qquad
f_\lambda^*(x,t)=\left\{
        \begin{tabular}{ll}
         $\overline{f}_{\lambda}(u_{\lambda_1})$&if $t\leq u_{\lambda_1}$, \\[1pt]
         $\overline{f}_{\lambda}(t)$&if $u_{\lambda_1}<t< u_{\lambda_2}$, \\[1pt]
         $\overline{f}_{\lambda}(u_{\lambda_2})$&if $t\geq u_{\lambda_2}$, 
        \end{tabular}
        \right.
\end{equation*}
and the energy functional 
\begin{equation*}
I_\lambda^*(u)=\frac{1}{2}\|u\|_{H_{\Sigma_{\mathcal{D}}}^s(\Omega)}^2-\int_{\Omega}F_\lambda^*(u)dx.
\end{equation*}
It is clear that $I_\lambda^*$ attains its global minimum at some $u_0\in H_{\Sigma_{\mathcal{D}}}^s(\Omega)$, that is 
\begin{equation}\label{ineq:min}
I_\lambda^*(u_0)\leq I_{\lambda}^*(u),\quad\text{for all }u\in H_{\Sigma_{\mathcal{D}}}^s(\Omega).
\end{equation}
Moreover,
\begin{equation*}
        \left\{
        \begin{tabular}{lcl}
        $(-\Delta)^su_0=f_\lambda^*(x,u_0)$, & &in  $\Omega$, \\
        $\mkern+22muB(u_0)=0$& &on  $\partial\Omega$.
        \end{tabular}
        \right.
\end{equation*}
Because of Theorem \ref{th_stMax} we get, for some $\varepsilon>0$,
\begin{equation*}
u_{\lambda_1}(x)+\varepsilon v(x)\leq u_0(x)\leq u_{\lambda_2}(x)-\varepsilon v(x),
\end{equation*}
for all $x\in\overline{\Omega}$, so that 
\begin{equation*}
u_{\lambda_1}<u_0<u_{\lambda_2},
\end{equation*}
for all $x\in\Omega$. Then, taking $ \left\| \frac{\omega-u_0}{v}\right\|_{L^{\infty}(\Omega)}\leq\tau$ with $0<\tau<\varepsilon$, we get $u_{\lambda_1}\leq\omega\leq u_{\lambda_2}$ in $\Omega$. Since $I_{\lambda}^*(\omega)-\overline{I}_{\lambda}(\omega)$ is zero for all $u_{\lambda_1}\leq\omega\leq u_{\lambda_2}$, by \eqref{ineq:min} we conclude 
\begin{equation*}
\overline{I}_\lambda(\omega)=I_\lambda^*(\omega)\geq I_\lambda^*(u_0)=\overline{I}_{\lambda}(u_0),
\end{equation*}
for all $\omega\in \mathfrak{C}_v(\Omega)$ with $\left\| \frac{\omega-u_0}{v}\right\|_{L^{\infty}(\Omega)}\leq\tau$ and the result follows.
\end{proof}
\begin{proposition}\label{prop:minC_minH}
Let $u_0\in H_{\Sigma_{\mathcal{D}}}^s(\Omega)$ be a local minimum of the energy functional $I_\lambda$ in $\mathfrak{C}_v(\Omega)$, i.e., there exists $r>0$ such that 
\begin{equation*}
I_{\lambda}(u_0)\leq I_{\lambda}(u_0+\omega),\quad \text{for all } \omega\in \mathfrak{C}_v(\Omega)\text{ with } \left\|\frac{\omega}{v}\right\|_{L^{\infty}(\Omega)}\leq r.
\end{equation*}
Then, $u_0$ is a local minimum of the energy functional $I_\lambda$ in $H_{\Sigma_{\mathcal{D}}}^s(\Omega)$, i.e., there exists $\delta>0$ such that 
\begin{equation*}
I_{\lambda}(u_0)\leq I_{\lambda}(u_0+\varphi),\quad \text{for all } \varphi\in H_{\Sigma_{\mathcal{D}}}^s(\Omega)\text{ with } \|\varphi\|_{H_{\Sigma_{\mathcal{D}}}^s(\Omega)}\leq \delta.
\end{equation*}
\end{proposition}
\begin{proof}
%The proof follows as in \cite[Theorem 5.1]{Colorado2003}. 
Arguing by contradiction we assume that, 
\begin{equation}\label{eq:contradiction}
\forall \varepsilon>0,\ \exists\, u_{\varepsilon}\in B_{\varepsilon} (u_0)
\ \text{such that } I_{\lambda}(u_{\varepsilon})<I_{\lambda}(u_0),
\end{equation}
where $\displaystyle B_{\varepsilon}(u_0)=\{u\in H_{\Sigma_{\mathcal{D}}}^s(\Omega): \|u-u_0\|_{H_{\Sigma_{\mathcal{D}}}^s(\Omega)}\leq \varepsilon\}$. For every $j>0$ let us consider the truncation map
\begin{equation*}
T_j(t)=\left\{
        \begin{tabular}{ll}
         $t$&if $0<t<j$, \\[1pt]
         $j$&if $t\geq j$,
        \end{tabular}
        \right.
\end{equation*}
let us set 
\begin{equation*}
F_{\lambda,j}(u)=\int_{0}^{u}f_{\lambda,j}(t)dt\qquad\text{where}\qquad
f_{\lambda,j}(t)=f_{\lambda}(T_j(t)),
\end{equation*}
with $f_\lambda(t)=\lambda|t|^{q-1}t+|t|^{2_s^*-1}t$, and the energy functional 
\begin{equation*}
I_{\lambda,j}(u)=\frac{1}{2}\|u\|_{H_{\Sigma_{\mathcal{D}}}^s(\Omega)}^2-\int_{\Omega}F_{\lambda,j}(u)dx.
\end{equation*}
For each $u\in H_{\Sigma_{\mathcal{D}}}^s(\Omega)$, we have $I_{\lambda,j}(u)\to I_{\lambda}(u)$ as $j\to\infty$. Hence, for any $\varepsilon>0$ there exists $j=j(\varepsilon)$ big enough such that $I_{\lambda,j(\varepsilon)}(u_{\varepsilon})<I_{\lambda}(u_0)$. Clearly, $\min_{B_{\varepsilon}(u_0)}I_{\lambda,j(\varepsilon)}$ is attained at some point, say $\omega_{\varepsilon}$. Therefore, we have
\begin{equation*}
I_{\lambda,j(\varepsilon)}(\omega_{\varepsilon})\leq I_{\lambda,j(\varepsilon)}(u_{\varepsilon})<I_{\lambda}(u_0).
\end{equation*}
Now we want to prove that 
\begin{equation}\label{eq:conv}
\omega_{\varepsilon}\to u_0\quad\text{in }\mathfrak{C}_v(\Omega)\text{ as }\varepsilon\to0, 
\end{equation}
and we will arrive to a contradiction with \eqref{eq:contradiction}. The Euler--Lagrange equation satisfied by $\omega_{\varepsilon}$ involves a Lagrange multiplier $\xi_{\varepsilon}$ such that
\begin{equation*}
\langle I_{\lambda,j(\varepsilon)}'(\omega_{\varepsilon}),\psi\rangle_{{\small H^{-s}(\Omega), H_{\Sigma_{\mathcal{D}}}^s(\Omega)}}=\xi_{\varepsilon}\langle \omega_{\varepsilon},\psi \rangle_{H_{\Sigma_{\mathcal{D}}}^s(\Omega)},\quad\forall\psi\in H_{\Sigma_{\mathcal{D}}}^s(\Omega),
\end{equation*}
that is, 
\begin{equation}\label{eq:multiplier}
\int_{\Omega}(-\Delta)^{\frac{s}{2}}\omega_{\varepsilon}(-\Delta)^{\frac{s}{2}}\psi dx-\int_{\Omega}f_{\lambda,j(\varepsilon)}(\omega_{\varepsilon})\psi dx=\xi_{\varepsilon}\int_{\Omega}(-\Delta)^{\frac{s}{2}}\omega_{\varepsilon}(-\Delta)^{\frac{s}{2}}\psi dx,
\end{equation}
for all $\psi\in H_{\Sigma_{\mathcal{D}}}^s(\Omega)$. Since $\omega_{\varepsilon}$ is a minimum of $I_{\lambda,j(\varepsilon)}$, it holds that, for $0<\varepsilon\ll1$,
\begin{equation}\label{eq:multiplier2}
\xi_{\varepsilon}=\frac{\langle I_{\lambda,j(\varepsilon)}'(\omega_{\varepsilon}),\omega_{\varepsilon}\rangle_{{\small H^{-s}(\Omega), H_{\Sigma_{\mathcal{D}}}^s(\Omega)}}}{\|\omega_{\varepsilon}\|_{H_{\Sigma_{\mathcal{D}}}^s(\Omega)}^2}\qquad\text{and}\qquad\xi_{\varepsilon}\to0\quad\text{as }\varepsilon\to0.
\end{equation}
Let us observe that, by \eqref{eq:multiplier}, the function $\omega_{\varepsilon}$ satisfies the problem
\begin{equation*}
        \left\{
        \begin{tabular}{lcl}
        $(-\Delta)^s\omega_{\varepsilon}=\frac{1}{1-\xi_{\varepsilon}}f_{\lambda,j(\varepsilon)}(\omega_{\varepsilon})\vcentcolon=f_{\lambda}^{\varepsilon}(\omega_{\varepsilon})$, & &in  $\Omega$, \\
        $\mkern+22muB(\omega_{\varepsilon})=0$& &on  $\partial\Omega$.
        \end{tabular}
        \right.
\end{equation*}
Clearly $\|\omega_{\varepsilon}\|_{H_{\Sigma_{\mathcal{D}}}^s(\Omega)}\leq C$. Moreover, by Proposition \ref{prop:bound}, we have $\|\omega_{\varepsilon}\|_{L^{\infty}(\Omega)}\leq C_1$ for a constant $C_1>0$ independent of $\varepsilon$. Thus, by \eqref{eq:multiplier2}, we get $\|f_{\lambda}^{\varepsilon}(\omega_{\varepsilon})\|_{L^{\infty}(\Omega)}\leq C_2$ so that \cite[Theorem 1.2]{Carmona2020} implies $\|\omega_{\varepsilon}\|_{\mathcal{C}^{\gamma}}\leq C_3$ for some $\gamma\in(0,\frac12)$ and a constant $C_3>0$ independent of $\varepsilon$. Here $\mathcal{C}^{\gamma}$ denotes the space of Hölder continuous functions with exponent $\gamma$. Then, by the Ascoli--Arzelá Theorem, there exist a subsequence, still denoted by $\omega_{\varepsilon}$, such that $\omega_{\varepsilon}\to u_0$ uniformly as $\varepsilon\to0$. Since $\omega_{\varepsilon}\to u_0$ in $H_{\Sigma_{\mathcal{D}}}^s(\Omega)$, we have $\omega_0=u_0$ and, by the Maximum Principle and Theorem \ref{th_stMax}, we get 
\begin{equation*}
\left\| \frac{\omega_{\varepsilon}-u_0}{v} \right\|_{L^{\infty}(\Omega)}\leq C \sup\limits_{\Omega}|f_{\lambda}^{\varepsilon}(\omega_{\varepsilon})-f_{\lambda}^{0}(u_0)|\to0\quad\text{as }\varepsilon\to0,
\end{equation*}
and \eqref{eq:conv} is proved.
\end{proof}
Propositions \ref{prop:min_C} and \ref{prop:minC_minH} provide us with a local minimum of $I_\lambda$ in $H_{\Sigma_{\mathcal{D}}}^s(\Omega)$ which will be denoted by $u_0$. Now, fixed $\lambda>0$, we look for a second solution of the form $\hat{u}=u_0+\tilde{u}$ with $u_0$ the above solution and $\tilde{u}>0$. The corresponding problem satisfied by $\tilde{u}$ is 
\begin{equation*}
        \left\{
        \begin{tabular}{lcl}
        $(-\Delta)^s\tilde{u}=\lambda(u_0+\tilde{u})^q-\lambda u_0^q+(u_0+\tilde{u})^{2_s^*-1}-u_0^{2_s^*-1}$ & &\mbox{in } $\Omega$, \\
        $\mkern+21mu B(\tilde{u})=0\mkern+0.5mu$& &\mbox{on } $\partial\Omega$,
        \end{tabular}
        \right.
\end{equation*}
Let us define the functions
\begin{equation*}
G_\lambda(u)=\int_{0}^{u}g_{\lambda}(x,t)dt\quad\text{with}\quad
g_\lambda(x,t)=\left\{
        \begin{tabular}{lr}
         $\lambda(u_0+t)^q-\lambda u_0^q+(u_0+t)^{2_s^*-1}-u_0^{2_s^*-1}$&\mbox{if } $t\geq0$, \\[5pt]
         $0$&\mbox{if } $t<0$,
        \end{tabular}
        \right.
\end{equation*}
and the energy functional 
\begin{equation*}
\widetilde{I}_\lambda(u)=\frac{1}{2}\|u\|_{H_{\Sigma_{\mathcal{D}}}^s(\Omega)}^2-\int_{\Omega}G_\lambda(u)dx.
\end{equation*}
Note that, since $u\in H_{\Sigma_{\mathcal{D}}}^s(\Omega)$, then $\widetilde{I}_\lambda$ is well defined. Finally let us define the moved problem 
\begin{equation}\label{p_tilde}
        \left\{
        \begin{tabular}{lcl}
        $(-\Delta)^su=g_\lambda(x,u)$ & &\mbox{in } $\Omega$, \\
        $\mkern+22mu B(u)=0\mkern+0.5mu$& &\mbox{on } $\partial\Omega$,
        \end{tabular}
        \right.
        \tag{$\widetilde{P}_\lambda$}
\end{equation}
Thus, if $\tilde{u}\not\equiv0$ is a critical point of $\widetilde{I}_\lambda$ then it is a solution to $(\widetilde{P}_\lambda)$. Moreover, by the Maximum Principle (cf. \cite[Lemma 2.3]{Capella2011}), we have $\tilde{u}>0$ in $\Omega$. Hence $\hat{u}=u_0+\tilde{u}$ will be a second solution to \eqref{p_lambda} with a concave nonlinearity $0<q<1$. Then, in order to prove the existence of a second solution we have to analyze the existence of nontrivial critical points of the functional $\widetilde{I}_\lambda$.

First we have the following.

\begin{lemma}\label{lem:min_0}
$u=0$ is a local minimum of $\widetilde{I}_\lambda$ in $H_{\Sigma_{\mathcal{D}}}^s(\Omega)$.
\end{lemma}

\begin{proof}
%The proof follows as in \cite[Lemma 4.2]{Ambrosetti1994} (see also \cite[Lemma 5.6]{Colorado2003}). 
Note that, because of  Proposition \ref{prop:minC_minH}, it is sufficient to prove that $u=0$ is a local minimum of $\widetilde{I}_\lambda$. Let $u^+$ be the positive part of $u$. Since
\begin{equation*}
G_{\lambda}(u^+)=\overline{F}_{\lambda}(u_0+u^+)-\overline{F}(u_0)-\left(\lambda u_0^q+u_0^{2_s^*-1}\right)u^+,
\end{equation*}
then
\begin{equation*}
\begin{split}
\widetilde{I}_\lambda(u)%=&\frac{1}{2}\|u^+\|_{H_{\Sigma_{\mathcal{D}}}^s(\Omega)}^{2}+\frac{1}{2}\|u^-\|_{H_{\Sigma_{\mathcal{D}}}^s(\Omega)}^{2}-\int_{\Omega}G_{\lambda}(u^+)\, dx\\
=&\frac{1}{2}\|u^+\|_{H_{\Sigma_{\mathcal{D}}}^s(\Omega)}^{2}+\frac{1}{2}\|u^-\|_{H_{\Sigma_{\mathcal{D}}}^s(\Omega)}^{2}\\
&-\int_{\Omega}\overline{F}_{\lambda}(u_0+u^+)\, dx+\int_{\Omega}\overline{F}(u_0)\, dx+\int_{\Omega}\left(\lambda u_0^q-u_0^{2_s^*-1}\right)u^+dx.
\end{split}
\end{equation*}
On the other hand
\begin{equation*}
\begin{split}
\overline{I}_\lambda(u_0+u^+)=& \frac{1}{2}\|u_0+u^+\|_{H_{\Sigma_{\mathcal{D}}}^s(\Omega)}^{2}-\int_{\Omega}\overline{F}_{\lambda}(u_0+u^+) dx\\
=&\frac{1}{2}\|u_0\|_{H_{\Sigma_{\mathcal{D}}}^s(\Omega)}^{2}+\frac{1}{2}\|u^+\|_{H_{\Sigma_{\mathcal{D}}}^s(\Omega)}^{2}+\int_{\Omega}(-\Delta)^{\frac{s}{2}}u_0(-\Delta)^{\frac{s}{2}}u^+dx\!-\!\int_{\Omega}\overline{F}_{\lambda}(u_0+u^+)dx\\
=& \frac{1}{2}\|u_0\|_{H_{\Sigma_{\mathcal{D}}}^s(\Omega)}^{2}+\frac{1}{2}\|u^+\|_{H_{\Sigma_{\mathcal{D}}}^s(\Omega)}^{2}+\int_{\Omega}\left(\lambda u_0^q+u_0^{2_s^*-1}\right)u^+dx\!-\!\int_{\Omega}\overline{F}_{\lambda}(u_0+u^+)dx.
\end{split}
\end{equation*}
Hence
\begin{equation*}
\begin{split}
\widetilde{I}_\lambda(u)=&\frac{1}{2}\|u^-\|_{H_{\Sigma_{\mathcal{D}}}^s(\Omega)}^{2}+\overline{I}_\lambda(u_0+u)-\frac{1}{2}\|u_0\|_{H_{\Sigma_{\mathcal{D}}}^s(\Omega)}^{2}+\int_{\Omega}\overline{F}_{\lambda}(u_0)dx\\
=&\frac{1}{2}\|u^-\|_{H_{\Sigma_{\mathcal{D}}}^s(\Omega)}^{2}+\overline{I}_\lambda(u_0+u)-\overline{I}_\lambda(u_0).
\end{split}
\end{equation*}
Using Proposition \ref{prop:min_C}, it follows that 
\begin{equation*}
\widetilde{I}_\lambda(u)\geq\frac{1}{2}\|u^-\|_{H_{\Sigma_{\mathcal{D}}}^s(\Omega)}^{2}\geq0,
\end{equation*}
provided $\displaystyle \left\| \frac{u}{v}\right\|_{L^{\infty}(\Omega)}\leq\varepsilon$ for some $\varepsilon>0$ small enough.
\end{proof}
As a consequence of Lemma \ref{lem:min_0} the following result holds for the energy functional associated to the extension problem corresponding to the moved problem $(\widetilde{P}_\lambda)$, namely,
\begin{equation*}
\widetilde{J}_{\lambda}(w)=\frac{1}{2}\|w\|_{\mathcal{X}_{\Sigma_{\mathcal{D}}}^s(\mathscr{C}_{\Omega})}^2-\int_{\Omega}G_\lambda(w(x,0))dx.
\end{equation*}
\begin{corollary}\label{cor:min_0_ext}
$w=0$ is a local minimum of $\widetilde{J}_\lambda$ in $\mathcal{X}_{\Sigma_{\mathcal{D}}}^s(\mathscr{C}_{\Omega})$.
\end{corollary}

\subsection{The Palais--Smale condition}\hfill\newline

Assuming that $w=0$ is the unique critical point of the energy functional $\widetilde{J}_{\lambda}$, corresponding to the extension problem of the moved problem $(\widetilde{P}_{\lambda})$, we prove that $\widetilde{J}_{\lambda}$ satisfies a local
Palais–Smale condition for energy levels $c$ below the critical level $c_{\mathcal{D-N}}^*=\frac{s}{N}[\widetilde{S}(\Sigma_{\mathcal{D}})]^{\frac{N}{2s}}$. As commented above, the main tool for proving this fact is the extension to the fractional-mixed setting of the concentration-compactness principle by Lions (cf. \cite{Lions1985}) provided by Theorem \ref{concompact}. We will also need some estimates for the action of the fractional Laplacian on the Sobolev extremal functions \eqref{eq:sob_extremal} in order to find paths with energy below $c_{\mathcal{D-N}}^*$.

\begin{lemma}\label{lem:PS}
If $w=0$ is the only critical point of $\widetilde{J}_{\lambda}$ in $\mathcal{X}_{\Sigma_{\mathcal{D}}}^s(\mathscr{C}_{\Omega})$, then $\widetilde{J}_{\lambda}$ satisfies a local Palis--Smale condition below the critical level
\begin{equation*}
c_{\mathcal{D-N}}^*=\frac{s}{N}[\widetilde{S}(\Sigma_{\mathcal{D}})]^{\frac{N}{2s}}.
\end{equation*}
\end{lemma}
\begin{proof}
Let $\{w_n\}$ be a PS sequence for $\widetilde{J}_{\lambda}$ verifying
\begin{equation}\label{lem:PS_eq0}
\widetilde{J}_{\lambda}(w_n)\to c< c_{\mathcal{D-N}}^*\quad\text{and}\quad \widetilde{J}_{\lambda}'(w_n)\to0.
\end{equation}
Then, it is clear that the sequence $\{w_n\}$ is uniformly bounded in $\mathcal{X}_{\Sigma_{\mathcal{D}}}^s(\mathscr{C}_{\Omega})$, say $\|w_n\|_{\mathcal{X}_{\Sigma_{\mathcal{D}}}^s(\mathscr{C}_{\Omega})}^2\leq M$ and, since by hypothesis $w=0$ is the unique critical point of $\widetilde{J}_{\lambda}$, it follows that 
\begin{equation*}
\begin{split}
w_n&\rightharpoonup 0\quad\text{weakly in }\mathcal{X}_{\Sigma_{\mathcal{D}}}^s(\mathscr{C}_{\Omega}),\\
w_n(x,0)&\to 0\quad\text{strongly in }L^r(\Omega),\ 1\leq r<2_s^*,\\
w_n(x,0)&\to 0\quad\text{a.e. in }\Omega.
\end{split}
\end{equation*}
Also, since $w_0=E_s[u_0]$ is a critical point of $J_{\lambda}$, we get 
\begin{equation}\label{lem:PS_eq1}
\widetilde{J}_{\lambda}(w_n)+J_\lambda(w_0)\geq J_\lambda(z_n)
\end{equation}
where $z_n=w_n+w_0$ and, then,
\begin{equation}\label{lem:PS_eq2}
J_\lambda(z_n)\to c+J_{\lambda}(w_0)\quad\text{and}\quad J'_{\lambda}(z_n)\to0.
\end{equation}
From \eqref{lem:PS_eq1} and \eqref{lem:PS_eq2} we get that the sequence $\{z_m\}$ is uniformly bounded in $\mathcal{X}_{\Sigma_{\mathcal{D}}}^s(\mathscr{C}_{\Omega})$, say $\|z_n\|_{\mathcal{X}_{\Sigma_{\mathcal{D}}}^s(\mathscr{C}_{\Omega})}^2\leq M_2$. This, together with the fact that $w=0$ is the unique critical point of $\widetilde{J}_{\lambda}$, up to a subsequence, we get 
\begin{equation}\label{lem:PS_eq1b}
\begin{split}
z_n&\rightharpoonup w_0\quad\text{weakly in }\mathcal{X}_{\Sigma_{\mathcal{D}}}^s(\mathscr{C}_{\Omega}),\\
z_n(x,0)&\to w_0(x,0)\quad\text{strongly in }L^r(\Omega),\ 1\leq r<2_s^*,\\
z_n(x,0)&\to w_0(x,0)\quad\text{a.e. in }\Omega.
\end{split}
\end{equation}
In order to apply the concentration-compactness result provided by Theorem \ref{concompact} we claim the following:\newline
\textbf{Claim:} The sequence $\{y^{1-2s}|\nabla z_n|^2\}_{n\in\mathbb{N}}$ is tight, i.e., for any $\eta>0$ there exists $\rho>0$ such that
\begin{equation*}
\int_{\{y>\rho\}}\int_{\Omega}y^{1-2s}|\nabla z_n|^2dxdy\leq\eta,\ \forall n\in\mathbb{N}.
\end{equation*}
The proof of the claim follows verbatim that of \cite[Lemma 3.6]{Barrios2012} (see also the proof of \cite[Theorem 4.5]{Colorado2019}) so we omit the details.

We can then apply Theorem \ref{concompact}. Hence, up to a subsequence, there exists an at most countable index set $\mathfrak{I}$, a sequence of points $\{x_i\}_{i\in\mathfrak{I}}\subset\overline{\Omega}$ and nonnegative real numbers $\mu_i$ and $\nu_i$ such that
\begin{equation}\label{lem:PS_eq3}
\kappa_s y^{1-2s}|\nabla z_n|^2\to \mu\geq\kappa_sy^{1-2s}|\nabla w_0|^2+\sum\limits_{i\in \mathfrak{I}}\mu_{i}\delta_{x_{i}},
\end{equation}
 and
 \begin{equation}\label{lem:PS_eq4}
 |z_n(x,0)|^{2_s^*}\to \nu=|w_0(x,0)|^{2_s^*}+\sum\limits_{i\in \mathfrak{I}}\nu_{i}\delta_{x_{i}},
 \end{equation}
 in the sense of measures and satisfying the relation 
 \begin{equation}\label{ineq:compac}
 \mu_i\geq \widetilde{S}(\Sigma_{\mathcal{D}})\nu_i^{\frac{2}{2_s^*}},\quad\text{for }i\in\mathfrak{I}.
 \end{equation}
Next, we fix $i_0\in \mathfrak{I}$ and we let $\phi\in\mathcal{C}_0^{\infty}(\mathbb{R}_+^{N+1})$ be a non-increasing smooth cut-off function verifying $\phi=1$ in $B_1^+(x_{i_0})$ and $\phi=0$ in $B_2^+(x_{i_0})^c$, with $B_r^+(x_{i_0})\subset\mathbb{R}^{N}\times\{y\geq0\}$ the $(N+1)$-dimensional semi-ball of radius $r>0$ centered at $x_{i_0}$. Let now $\phi_{\varepsilon}(x,y)=\phi(x/\varepsilon,y/\varepsilon)$, clearly $|\nabla\phi_{\varepsilon}|\leq\frac{C}{\varepsilon}$ and let us denote by $\Gamma_{2\varepsilon}=B_{2\varepsilon}^+(x_{i_0})\cap\{y=0\}$. Then, taking the dual product in \eqref{lem:PS_eq2} with $\phi_{\varepsilon}z_n$, we have
\begin{equation*}
\begin{split}
\lim_{n\to\infty}&\kappa_s\int_{\mathscr{C}_{\Omega}}y^{1-2s}\langle\nabla z_n,\nabla\phi_{\varepsilon}\rangle z_ndxdy\\
&=\lim_{n\to\infty} \left(\lambda\int_{\Gamma_{2\varepsilon}}|z_n|^{q+1}\phi_{\varepsilon}dx+\int_{\Gamma_{2\varepsilon}}|z_n|^{2_s^*}\phi_{\varepsilon}dx-\kappa_s\int_{B_{2\varepsilon}^+(x_{i_0})}y^{1-2s}|\nabla z_n|^2\phi_{\varepsilon}dxdy\right)
\end{split}
\end{equation*}
Then, thanks to \eqref{lem:PS_eq1b}, \eqref{lem:PS_eq3} and \eqref{lem:PS_eq4}, we find,
\begin{equation}\label{eq:tozero}
\begin{split}
\lim_{n\to\infty}&\kappa_s\int_{\mathscr{C}_{\Omega}}y^{1-2s}\langle\nabla z_n,\nabla\phi_{\varepsilon}\rangle z_ndxdy\\
&=\lambda\int_{\Gamma_{2\varepsilon}}|w_0|^{q+1}\phi_{\varepsilon}dx+\int_{\Gamma_{2\varepsilon}}\phi_{\varepsilon}d\nu-\kappa_s\int_{B_{2\varepsilon}^+(x_{i_0})}y^{1-2s}\phi_{\varepsilon}d\mu.
\end{split}
\end{equation}
Assume for the moment that the left hand side of \eqref{eq:tozero} vanishes as $\varepsilon\to0$. Thus, 
\begin{equation}\label{compacidad}
0=\lim_{\varepsilon\to0}\left(\lambda\int_{\Gamma_{2\varepsilon}}|w_0|^{q+1}\phi_{\varepsilon}dx+\int_{\Gamma_{2\varepsilon}}\phi_{\varepsilon}d\nu-\kappa_s\int_{B_{2\varepsilon}^+(x_{i_0})}y^{1-2s}\phi_{\varepsilon}d\mu\right)=\nu_{i_0}-\mu_{i_0}.
\end{equation}
We have then two options, either the compactness of the PS sequence or concentration around the point $x_{i_0}$. In other words, either $\nu_{i_0}=0$, so that $\mu_{i_0}=0$ or, by \eqref{compacidad} and \eqref{ineq:compac}, $$\nu_{i_0}\geq [\widetilde{S}(\Sigma_{\mathcal{D}})]^{\frac{N}{2s}}.$$ In case of having concentration, we find,
\begin{equation*}
\begin{split}
c+J_{\lambda}(w_0)=&\lim_{n\to\infty}J_{\lambda}(z_n)-\frac{1}{2}\langle J_{\lambda}'(z_n),z_n\rangle\\
\geq&\lambda\left(\frac{1}{2}-\frac{1}{q+1}\right)\int_{\Omega}|w_0|^{q+1}dx+\frac{s}{N}\int_{\Omega}|w_0|^{2_s^*}dx+\frac{s}{N}\nu_{i_0}\\
\geq&J_{\lambda}(w_0)+\frac{s}{N}[\widetilde{S}(\Sigma_{\mathcal{D}})]^{\frac{N}{2s}}=J_{\lambda}(w_0)+c_{\mathcal{D-N}}^*,
\end{split}
\end{equation*}
in contradiction with the hypotheses $c<c_{\mathcal{D-N}}^*$. Since $i_0$ was chosen arbitrarily, $\nu_i=0$ for all $i\in\mathfrak{I}$. As a consequence $u_n\to u_0$ in $L^{2_s^*}(\Omega)$. Since convergence of $u_n$ in $L^{2_s^*}(\Omega)$ implies convergence of $f_{\lambda}(u_n)$ in $L^{\frac{2N}{N+2s}}(\Omega)$ using the continuity of the inverse operator $(-\Delta)^{-s}$ we get the strong convergence of $u_n$ in $H_{\Sigma_{\mathcal{D}}}^s(\Omega)$.

It only remains to prove that the left hand side of \eqref{eq:tozero} vanishes as $\varepsilon\to0$. The PS sequence $\{z_m\}$ is uniformly bounded in $\mathcal{X}_{\Sigma_{\mathcal{D}}}^s(\mathscr{C}_{\Omega})$ and, up to a subsequence, $z_n\rightharpoonup w_0\in \mathcal{X}_{\Sigma_{\mathcal{D}}}^s(\mathscr{C}_{\Omega})$ and $z_n\rightarrow w_0$ a.e. in $\mathscr{C}_{\Omega}$. Moreover, for $r<2^*=\frac{2(N+1)}{N-1}$ we have the compact inclusion $\mathcal{X}_{\Sigma_{\mathcal{D}}}^s(\mathscr{C}_{\Omega})\hookrightarrow L^{r}(\mathscr{C}_{\Omega},y^{1-2\beta}dxdy)$. Then, applying H\"older's inequality with $p=\frac{N+1}{N-1}$ and $q=\frac{N+1}{2}$, we find,
\begin{equation*}
\begin{split}
\int_{B_{2\varepsilon}^+(x_{i_0})}  &y^{1-2s}|\nabla\phi_{\varepsilon}|^2|z_n|^2dxdy\\
\leq& \left(\int_{B_{2\varepsilon}^+(x_{i_0})}y^{1-2s}|\nabla\phi_{\varepsilon}|^{N+1}dxdy\right)^{\frac{2}{N+1}}\left(\int_{B_{2\varepsilon}^+(x_{i_0})} y^{1-2s}|z_n|^{2\frac{N+1}{N-1}}dxdy\right)^{\frac{N-1}{(N+1)}}\\
\leq&\frac{1}{\varepsilon^2}\left(\int_{B_{2\varepsilon}(x_{i_0})}\int_0^\varepsilon y^{1-2s}dxdy\right)^{\frac{2}{N+1}}\left(\int_{B_{2\varepsilon}^+(x_{i_0})} y^{1-2s}|z_n|^{2\frac{N+1}{N-1}}dxdy\right)^{\frac{N-1}{(N+1)}}\\
\leq& c_0\varepsilon^{\frac{2(1-2s)}{N+1}}\left(\int_{B_{2\varepsilon}^+(x_{i_0})}y^{1-2s}|z_n|^{2\frac{N+1}{N-1}}dxdy\right)^{\frac{N-1}{(N+1)}}\\
\leq& c_0 \varepsilon^{\frac{2(1-2s)}{N+1}}\varepsilon^{\frac{(2+N-2s)(N-1)}{(N+1)}}\left(\int_{B_{2}^+(x_{i_0})}y^{1-2s}|z_n(\varepsilon x,\varepsilon y)|^{2\frac{N+1}{N-1}}dxdy\right)^{\frac{N-1}{(N+1)}}\\
\leq& c_1 \varepsilon^{N-2s},
\end{split}
\end{equation*}
for appropriate positive constants $c_0$ and $c_1$. 
Thus, we conclude,
\begin{equation*}
\begin{split}
0\leq&\lim_{n\to\infty}\left|\kappa_s\int_{\mathscr{C}_{\Omega}}y^{1-2s}\langle\nabla z_n,\nabla\phi_{\varepsilon} \rangle z_ndxdy\right|\\
\leq&\kappa_s\lim_{n\to\infty}\left(\int_{\mathscr{C}_{\Omega}}y^{1-2s}|\nabla z_n|^2dxdy\right)^{1/2}\left(\int_{B_{2\varepsilon}^+(x_{k_0})}y^{1-2s}|\nabla\phi_{\varepsilon}|^2|z_n|^2dxdy\right)^{1/2}\\
\leq&C \varepsilon^{\frac{N-2s}{2}}\to0,
\end{split}
\end{equation*}
as $\varepsilon\to 0$ and the proof of the Lemma \ref{lem:PS} is complete. 
\end{proof}

In Lemma \ref{lem:PS} we have proved that, if $w=0$ is the only critical point of the functional $\widetilde{J}_{\lambda}$, then $\widetilde{J}_{\lambda}$ verifies the Palais-Smale condition at any level $c<c^*_{\mathcal{D-N}}$. Next, we have to show that we can obtain a local PS$_c$ sequence for the energy functional $\widetilde{J}_{\lambda}$ with energy level $c<c^*_{\mathcal{D-N}}$. To do that we will use the family of minimizers $u_{\varepsilon}$ of the fractional Sobolev inequality \eqref{sobolev} given by \eqref{eq:sob_extremal} and its $s$-harmonic extension $w_{\varepsilon}=E_s[u_{\varepsilon}]$.

Consider a smooth non-increasing cut-off function $\phi_0(t)\in\mathcal{C}^{\infty}(\mathbb{R}_+)$, satisfying $\phi_0(t)=1$ for $0\leq t\leq\frac{1}{2}$ and $\phi_0(t)=0$ for $t\geq1$, and $|\phi_0'(t)|\le C$ for any $t\ge 0$. Assume, without loss of generality, that $0\in\Omega$, and define, for some $\rho>0$
small enough such that $B_{\rho}^+((0,0))\subseteq{\mathscr{C}}_{\Omega}$, the function $\phi(x,y)=\phi_{\rho}(x,y)=\phi_0(\frac{r_{xy}}{\rho})$ with $r_{xy}=|(x,y)|=(|x|^2+y^2)^{\frac{1}{2}}$. Then, we have the following.
\begin{lemma}\cite[Lemma 3.3]{Colorado2019}\label{lemma_est} The family $\{\phi_{\rho} w_{\varepsilon}\}$ and its trace on $\{y=0\}$, namely $\{\phi_{\rho} u_{\varepsilon}\}$, satisfy
\begin{equation*}
\|\phi_\rho w_{\varepsilon}\|_{\mathcal{X}_{\Sigma_{\mathcal{D}}}^s(\mathscr{C}_{\Omega})}^2=\|w_{\varepsilon}\|_{\mathcal{X}_{\Sigma_{\mathcal{D}}}^s(\mathscr{C}_{\Omega})}^2+ O\left(\left(\frac{\varepsilon}{\rho}\right)^{N-2s}\right)
\end{equation*}
and
\begin{equation*}
\int_{\Omega}|\phi_\rho
u_{\varepsilon}|^{2_s^*}dx=\|u_{\varepsilon}\|_{L^{2_s^*}(\mathbb{R}^N)}^{2_s^*}+O\left(\left(\frac{\varepsilon}{\rho}\right)^N\right).
\end{equation*}
for $\varepsilon>0$ small enough.
\end{lemma}
We consider now a cut-off function centered at $x_0\in\Sigma_{\mathcal{N}}$, namely, $\psi(x,y)=\psi_{\rho}(x,y)=\phi_0(\frac{\overline{r}_{xy}}{\rho})$ with $\overline{r}_{xy}=|(x-x_0,y)|=(|x-x_0|^2+y^2)^{\frac{1}{2}}$ and $\phi_0(t)$ as the cut-off function of Lemma \ref{lemma_est}. Then, in the spirit of \cite[Lemma 3.2]{Grossi1990} the following holds.

\begin{lemma}\label{lemma_est_boundary}
The family $\{\psi_{\rho} w_{\varepsilon}\}$ and its trace on $\{y=0\}$, namely $\{\psi_{\rho} u_{\varepsilon}\}$, satisfy
\begin{equation}\label{est1}
\|\psi_\rho w_{\varepsilon}\|_{\mathcal{X}_{\Sigma_{\mathcal{D}}}^s(\mathscr{C}_{\Omega})}^2=\frac{1}{2}\|w_{\varepsilon}\|_{\mathcal{X}_{\Sigma_{\mathcal{D}}}^s(\mathscr{C}_{\Omega})}^2+ O\left(\left(\frac{\varepsilon}{\rho}\right)^{N-2s}\right).
\end{equation}
and
\begin{equation}\label{est1b}
\|\psi_\rho u_{\varepsilon}\|_{L^{2_s^*}(\Omega)}^{2_s^*}=\frac{1}{2}\|u_{\varepsilon}\|_{L^{2_s^*}(\mathbb{R}^N)}^{2_s^*}+O\left(\left(\frac{\varepsilon}{\rho}\right)^N\right).
\end{equation}
for $\varepsilon>0$ small enough.
\end{lemma}
\begin{proof}
Take $X_0=(x_0,0)$ and denote by $\Omega_{\rho}$ the set $\Omega\cap B_{\rho}(x_0)$ and by $\mathscr{C}_\rho$ the set $\mathscr{C}_{\Omega}\cap B_{\rho}^+(X_0)$. Let us estimate the norm of the truncated function centered at $x_0\in\Sigma_{\mathcal{N}}$,
\begin{equation}\label{st_lem}
\begin{split}
\|\psi_\rho w_{\varepsilon}\|_{\mathcal{X}_{\Sigma_{\mathcal{D}}}^s(\mathscr{C}_{\Omega})}^2&=\kappa_s\int_{\mathscr{C}_{\Omega}}y^{1-2s}\psi_{\rho}^2|\nabla w_{\varepsilon}|^2dxdy\\
&+\kappa_s\int_{\mathscr{C}_{\Omega}}y^{1-2s}|w_{\varepsilon}\nabla\psi_{\rho}|^2dxdy+2\kappa_s\int_{\mathscr{C}_{\Omega}}y^{1-2s}\langle w_{\varepsilon}\nabla\psi_{\rho},\psi_{\rho}\nabla w_{\varepsilon} \rangle dxdy. 
\end{split}
\end{equation}
Following verbatim the proof of Lemma \ref{lemma_est}, the next estimate holds,
\begin{equation*}
\kappa_s\int_{\mathscr{C}_{\Omega}}y^{1-2s}|w_{\varepsilon}\nabla\psi_{\rho}|^2dxdy+\kappa_s \int_{\mathscr{C}_{\Omega}}y^{1-2s}\langle w_{\varepsilon}\nabla\psi_{\rho},\psi_{\rho}\nabla w_{\varepsilon} \rangle dxdy=O\left(\left(\frac{\varepsilon}{\rho}\right)^{N-2s}\right).
\end{equation*}
To estimate the remaining term in the right-hand side of \eqref{st_lem}, we use that, by \cite[Lemma 3.1]{Barrios2012}, 
\begin{equation*}
|\nabla w_{1,s}(x,y)|\leq C w_{1,s-\frac12}(x,y),\ \frac{1}{2}<s<1,\ (x,y)\in\mathbb{R}_+^{N+1},
\end{equation*}
and, for $(x,y)\in B_{\frac{\rho}{\varepsilon}}^+(X_0)$, we have $\displaystyle w_1(x,y)=O\left(\left(\frac{\varepsilon}{\rho}\right)^{N-2s}\right)$, so that
\begin{equation*}
\int_{\{r_{xy}\leq \rho\}}y^{1-2s}|\nabla w_{\varepsilon}|^2dxdy\leq C\int\limits_{\{r_{xy}\leq \frac{\rho}{\varepsilon}\}}y^{1-2s}\left(\frac{\varepsilon}{\rho}\right)^{2(N-2(s-1/2))}dxdy =O\left(\left(\frac{\varepsilon}{\rho}\right)^{N-2s}\right).
\end{equation*}
Then, since $\Omega$ is a smooth domain, by the above estimates we find
\begin{equation*}
\begin{split}
\int_{\mathscr{C}_{\Omega}}y^{1-2s}\psi_{\rho}^2|\nabla w_{\varepsilon}|^2dxdy&=\int_{\mathscr{C}_\rho}y^{1-2s}\psi_{\rho}^2|\nabla w_{\varepsilon}|^2dxdy=\int_{\mathscr{C}_\rho}y^{1-2s}|\nabla w_{\varepsilon}|^2dxdy+O\left(\left(\frac{\varepsilon}{\rho}\right)^{N-2s}\right)\\
&=\frac{1}{2}\int_{B_\rho^+(0)}y^{1-2s}|\nabla w_{\varepsilon}|^2dxdy+O\left(\left(\frac{\varepsilon}{\rho}\right)^{N-2s}\right),
\end{split}
\end{equation*}
Then, the estimate \eqref{est1} follows by applying Lemma \ref{lemma_est}.
%\begin{equation*}
%\|\psi_\rho w_{\varepsilon}\|_{\mathcal{X}_{\Sigma_{\mathcal{D}}}^s(\mathscr{C}_{\Omega})}^2=\frac{1}{2}\|w_{\varepsilon}\|_{\mathcal{X}_{\Sigma_{\mathcal{D}}}^s(\mathscr{C}_{\Omega})}^2+ O\left(\left(\frac{\varepsilon}{\rho}\right)^{N-2s}\right).
%\end{equation*}
On the other hand, since 
\begin{equation*}
\int_{|x|<\rho}|u_{\varepsilon}|^{2_s^*}dx=\int_{|x|<\rho}\frac{\varepsilon^N}{(\varepsilon^2+|x|^2)^N}dx=O\left(\left(\frac{\varepsilon}{\rho}\right)^{N}\right),
\end{equation*}
we have 
\begin{equation*}
\begin{split}
\int_{\Omega}|\psi_\rho u_{\varepsilon}|^{2_s^*}dx&=\int_{\Omega_\rho}|\psi_\rho u_{\varepsilon}|^{2_s^*}dx=\int_{\Omega_\rho}|u_{\varepsilon}|^{2_s^*}dx+O\left(\left(\frac{\varepsilon}{\rho}\right)^{N}\right)\\
&=\frac{1}{2}\int_{B_\rho(0)}| u_{\varepsilon}|^{2_s^*}+O\left(\left(\frac{\varepsilon}{\rho}\right)^{N}\right)dx.
\end{split}
\end{equation*}
Using Lemma \ref{lemma_est} we get \eqref{est1b}.
%\begin{equation*}
%\|\psi_\rho u_{\varepsilon}\|_{L^{2_s^*}(\Omega)}^{2_s^*}=\frac{1}{2}\|u_{\varepsilon}\|_{L^{2_s^*}(\mathbb{R}^N)}^{2_s^*}+O\left(\left(\frac{\varepsilon}{\rho}\right)^N\right).
%\end{equation*}
\end{proof}

Next, let us consider the family of functions 
\begin{equation}\label{eq:eta}
\eta_{\varepsilon}=\frac{\psi_{\rho} w_{\varepsilon}}{\|\psi_{\rho} u_{\varepsilon}\|_{L^{2_s^*}(\Omega)}}.
\end{equation}
with $\rho=\varepsilon^{\alpha}$ and $0<\alpha<1$ to be chosen (see \eqref{alpha_condition} below). Because of Lemma \ref{lemma_est_boundary}, we get
\begin{equation}\label{eq:trunc_sobolev}
\begin{split}
\|\eta_{\varepsilon}\|_{\mathcal{X}_{\Sigma_{\mathcal{D}}}^s(\mathscr{C}_{\Omega})}^2&
%=\frac{\|\psi_\rho w_{\varepsilon}\|_{\mathcal{X}_{\Sigma_{\mathcal{D}}}^s(\mathscr{C}_{\Omega})}^2}{\|\psi_{\rho} u_{\varepsilon}\|_{L^{2_s^*}(\Omega)}^2}
= \frac{\frac{1}{2}\|w_{\varepsilon}\|_{\mathcal{X}_{\Sigma_{\mathcal{D}}}^s(\mathscr{C}_{\Omega})}^2+O(\varepsilon^{(1-\alpha)(N-2s)})}{\left(\frac{1}{2}\right)^{\frac{2}{2_s^*}}\|u_{\varepsilon}\|_{L^{2_s^*}(\mathbb{R}^N)}^{2}+O(\varepsilon^{(1-\alpha)N})}\leq2^{-\frac{2s}{N}}\frac{\|w_{\varepsilon}\|_{\mathcal{X}^s(\mathbb{R}_+^{N+1})}^2+O(\varepsilon^{(1-\alpha)(N-2s)})}{\|u_{\varepsilon}\|_{L^{2_s^*}(\mathbb{R}^N)}^{2}+O(\varepsilon^{(1-\alpha)N})}\\
& \leq2^{-\frac{2s}{N}}S(N,s)+O(\varepsilon^{(1-\alpha)(N-2s)}).
\end{split}
\end{equation}
\begin{lemma}\cite[Lemma 3.8]{Barrios2012}\label{lemma_est_L2} The family $\{\phi_{\rho} w_{\varepsilon}\}$ and its trace on $\{y=0\}$, namely $\{\phi_{\rho} u_{\varepsilon}\}$, satisfy
\begin{equation*}
\|\phi_\rho u_{\varepsilon}\|_{L^2(\Omega)}^2=\left\{
        \begin{tabular}{lr}
        $C\varepsilon^{2s}+O(\varepsilon^{N-2s})$ & if $N>4s$, \\
        $C\varepsilon^{2s}\log(1/\varepsilon)+O(\varepsilon^{2s})$  & if $N=4s$, \\
        \end{tabular}
        \right.
\end{equation*}
and
\begin{equation*}
\|\phi_\rho u_{\varepsilon} \|_{L^{2_s^*-1}(\Omega)}^{2_s^*-1}\geq C\varepsilon^{\frac{N-2s}{2}},\quad \text{if }N>2s,
\end{equation*}
for $\varepsilon>0$ small enough.
\end{lemma}
\begin{remark}
In this case the dependence on the cut-off radius can be neglected since
\begin{equation*}
\begin{split}
\|\phi_\rho u_{\varepsilon}\|_{L^2(\Omega)}^2&=\int_{\Omega}|\phi_\rho u_\varepsilon|^2dx\geq\int_{\{|x|<\frac{\rho}{2}\}}|u_\varepsilon|^2dx=\varepsilon^{-(N-2s)}\int_{\{|x|<\frac{\rho}{2}\}}\left|u_1\left(\frac{x}{\varepsilon}\right)\right|^2dx\\
&=\varepsilon^{2s}\int_{\{|x|<\frac{\rho}{2\varepsilon}\}}|u_1(x)|^2dx,
\end{split}
\end{equation*}
so that, for $\rho=\varepsilon^\alpha$ with $0<\alpha<1$ as above, we have
\begin{equation*}
\|\phi_\rho u_{\varepsilon}\|_{L^2(\Omega)}^2\geq\varepsilon^{2s}\int_{\{|x|<\varepsilon^{-(1-\alpha)}\}}|u_1(x)|^2dx\geq\varepsilon^{2s}\int_{\{|x|<1\}}|u_1(x)|^2dx,
\end{equation*}
for $\varepsilon<<1$. Moreover, noticing that
\begin{equation*}
\varepsilon^{2s}\int_{\{|x|>\frac{\rho}{2\varepsilon}\}}|u_1|^2dx=\varepsilon^{2s}\cdot O\left(\left(\frac{\varepsilon}{\rho}\right)^{N-4s}\right),
\end{equation*}
takin $\rho=\varepsilon^\alpha$ as above we have $\|\phi_\rho u_{\varepsilon}\|_{L^2(\Omega)}^2=\varepsilon^{2s}\left(\| u_1\|_{L^2(\mathbb{R}^N)}^2+O\left(\varepsilon^{(1-\alpha)(N-4s)}\right)\right)$. Then, it follows that also $\|\psi_\rho u_{\varepsilon}\|_{L^2(\Omega)}^2=O(\varepsilon^{2s})$.
\end{remark}
We finally prove that we can find paths with energy below the critical level $c^*_{\mathcal{D-N}}$. As commented above, contrary to the fractional Sobolev constant \eqref{sobolev}, the constant $\widetilde{S}(\Sigma_{\mathcal{D}})$ is attained for $|\Sigma_{\mathcal{D}}|$ small enough. Indeed, as by Proposition \ref{prop_cota} we have $\displaystyle \widetilde{S}(\Sigma_{\mathcal{D}})\leq 2^{\frac{-2s}{N}}S(N,s)$,  we have two options:
\begin{itemize}
\item[1)]$\displaystyle \widetilde{S}(\Sigma_{\mathcal{D}})<2^{\frac{-2s}{N}}S(N,s)$. In this case, the constant $\displaystyle \widetilde{S}(\Sigma_{\mathcal{D}})$ is attained by Theorem \ref{th_att} and no concentration occurs.\vspace{0.2cm}
\item[2)]$\displaystyle \widetilde{S}(\Sigma_{\mathcal{D}})=2^{\frac{-2s}{N}}S(N,s)$. In this case, by Theorem \ref{CCA}, minimizing sequences for $\displaystyle \widetilde{S}(\Sigma_{\mathcal{D}})$ concentrate at a boundary point $x_0\in\overline{\Sigma}_{\mathcal{N}}$.
\end{itemize}
%%%%%%%%%%%%%%%%%%%%%%%%%%%%%%%%%
\begin{lemma}\label{lem:below_level0} 
Assume $\displaystyle \widetilde{S}(\Sigma_{\mathcal{D}})<2^{\frac{-2s}{N}}S(N,s)$. Then, there exists $\tilde{w}\in \mathcal{X}_{\Sigma_{\mathcal{D}}}^s(\mathscr{C}_{\Omega})$, $\tilde{w}>0$ such that
\begin{equation*}
\sup\limits_{t\geq0}\widetilde{J}_{\lambda}(t\tilde{w})<c^*_{\mathcal{D-N}}.
\end{equation*}
\end{lemma}
\begin{proof}
If $\displaystyle \widetilde{S}(\Sigma_{\mathcal{D}})<2^{\frac{-2s}{N}}S(N,s)$ no concentration occurs and $\widetilde{S}(\Sigma_{\mathcal{D}})$ is attained at some $\tilde{u}_{}\in H_{\Sigma_{\mathcal{D}}}^s(\Omega)$ that we can assume to be positive (cf. \cite{Colorado2019}). Take $\displaystyle\tilde{w}=\frac{E_s[\tilde{u}]}{\|\tilde{u}\|_{2_s^*}}$ so that, by \eqref{norma2},
\begin{equation*}
\|\tilde{w}\|_{\mathcal{X}_{\Sigma_{\mathcal{D}}}^s(\mathscr{C}_{\Omega})}^2=\widetilde{S}(\Sigma_{\mathcal{D}}).
\end{equation*}
Then, as $\lambda>0$ and $\tilde{w}>0$,
\begin{equation*}
\begin{split}
\widetilde{J}_{\lambda}(t\tilde{w})&=\frac{t^2}{2}\|\tilde{w}\|_{\mathcal{X}_{\Sigma_{\mathcal{D}}}^s(\mathscr{C}_{\Omega})}^2-\frac{\lambda}{q+1}\|\tilde{w}(x,0)\|_{L^{q+1}(\Omega)}^{q+1}-\frac{t^{2_s^*}}{2_s^*}\|\tilde{w}(x,0)\|_{L^{2_s^*}(\Omega)}^{2_s^*}\\
&=\left(\frac{t^2}{2}\widetilde{S}(\Sigma_{\mathcal{D}})-\frac{\lambda}{q+1}\|\tilde{w}(x,0)\|_{L^{q+1}(\Omega)}^{q+1}-\frac{t^{2_s^*}}{2_s^*}\right)\\
&<\left(\frac{t^2}{2}\widetilde{S}(\Sigma_{\mathcal{D}})-\frac{t^{2_s^*}}{2_s^*}\right)=\vcentcolon  g(t).
\end{split}
\end{equation*}
Since $g(t)$ attains its maximum at $t_0=[\widetilde{S}(\Sigma_{\mathcal{D}})]^{\frac{1}{2_s^*-2}}$ and $g(t_0)=[\widetilde{S}(\Sigma_{\mathcal{D}})]^{\frac{N}{2s}}$, we conclude
\begin{equation*}
\sup\limits_{t\geq0}\widetilde{J}_{\lambda}(t\tilde{w})<[\widetilde{S}(\Sigma_{\mathcal{D}})]^{\frac{N}{2s}}=c^*_{\mathcal{D-N}}.
\end{equation*}
\end{proof}
Next we address the case $\displaystyle \widetilde{S}(\Sigma_{\mathcal{D}})=2^{\frac{-2s}{N}}S(N,s)$. Since now the constant $\widetilde{S}(\Sigma_{\mathcal{D}})$ is not attained, we need to use the family of functions $\eta_{\varepsilon}$ defined in \eqref{eq:eta} and based on the extremal functions of the Sobolev inequality to construct paths below the critical level $c^*_{\mathcal{D-N}}=\frac{s}{N}[\widetilde{S}(\Sigma_{\mathcal{D}})]^{\frac{N}{2s}}=\frac{s}{2N}[S(N,s)]^{\frac{N}{2s}}$.
%%%%%%%%%%%%%%%%%%%%%%%%%%%%%%%%%
\begin{lemma}\label{lem:below_level} 
Assume $\displaystyle \widetilde{S}(\Sigma_{\mathcal{D}})=2^{\frac{-2s}{N}}S(N,s)$. Then, there exists $\varepsilon>0$ small enough such that 
\begin{equation*}
\sup\limits_{t\geq0}\widetilde{J}_{\lambda}(t\eta_{\varepsilon})<c^*_{\mathcal{D-N}}.
\end{equation*}
\end{lemma}
\begin{proof}
Assume $N\geq4s$. Then, using as before that 
\begin{equation}\label{ineq:ab}
(a+b)^p\geq a^p+b^p+\mu a^{p-1}b,\quad\text{for } a,b\geq0,\ p>1,\ \text{and some }\mu=\mu(p)>0,
\end{equation}
we have 
\begin{equation*}
G_{\lambda}(w)\geq\frac{1}{2_s^*}w^{2_s^*}+\frac{\mu}{2}w^2w_0^{2_s^*-2}\qquad\text{and}\qquad
g_{\lambda}(x,t)\geq t^{2_s^*-1}+\mu w_0^{2_s^*-2}.
\end{equation*}
Hence,
\begin{equation*}
\widetilde{J}_{\lambda}(t\eta_{\varepsilon})\leq \frac{t^2}{2}\|\eta_{\varepsilon}\|_{\mathcal{X}_{\Sigma_{\mathcal{D}}}^s(\mathscr{C}_{\Omega})}^2-\frac{t^{2_s^*}}{2_s^*}-\frac{t^2}{2}\mu\int_{\Omega}w_0^{2_s^*-2}\eta_{\varepsilon}^2dx.
\end{equation*}
Since $w_0\geq a_0>0$ in $supp(\eta_{\varepsilon})$, we get
\begin{equation*}
\widetilde{J}_{\lambda}(t\eta_{\varepsilon})\leq \frac{t^2}{2}\|\eta_{\varepsilon}\|_{\mathcal{X}_{\Sigma_{\mathcal{D}}}^s(\mathscr{C}_{\Omega})}^2-\frac{t^{2_s^*}}{2_s^*}-\frac{t^2}{2}\widetilde{\mu}\|\eta_{\varepsilon}\|_{L^2(\Omega)}^2=\vcentcolon g_{\varepsilon}(t).
\end{equation*}
It is clear that $\lim\limits_{t\to+\infty}g_{\varepsilon}(t)=-\infty$ and $\sup\limits_{t\geq 0}g_{\varepsilon}(t)$ is attained at some $t_{\varepsilon}$. In particular, since
\begin{equation*}
0=g_{\varepsilon}'(t_{\varepsilon})=t_{\varepsilon}\|\eta_{\varepsilon}\|_{\mathcal{X}_{\Sigma_{\mathcal{D}}}^s(\mathscr{C}_{\Omega})}^2-t_{\varepsilon}^{2_s^*-1}-t_{\varepsilon}\widetilde{\mu}\|\eta_{\varepsilon}\|_{L^2(\Omega)},
\end{equation*}
we have
\begin{equation*}
t_{\varepsilon}\leq \|\eta_{\varepsilon}\|_{\mathcal{X}_{\Sigma_{\mathcal{D}}}^s(\mathscr{C}_{\Omega})}^{\frac{2}{2_s^*-2}}.
\end{equation*}
Leet us note that, by Lemma \ref{lemma_est_boundary}, $t_{\varepsilon}\geq C>0$. On the other hand, the function 
\begin{equation*}
h_{\varepsilon}(t)=\frac{t^2}{2}\|\eta_{\varepsilon}\|_{\mathcal{X}_{\Sigma_{\mathcal{D}}}^s(\mathscr{C}_{\Omega})}^2-\frac{t^{2_s^*}}{2_s^*},
\end{equation*}
is increasing on $[0, \|\eta_{\varepsilon}\|_{\mathcal{X}_{\Sigma_{\mathcal{D}}}^s(\mathscr{C}_{\Omega})}^{\frac{2}{2_s^*-2}}]$. Therefore,
\begin{equation*}
\sup\limits_{t\geq 0}g_{\varepsilon}(t)=g_{\varepsilon}(t_{\varepsilon})\leq\frac{s}{N}\|\eta_{\varepsilon}\|_{\mathcal{X}_{\Sigma_{\mathcal{D}}}^s(\mathscr{C}_{\Omega})}^{\frac{N}{s}}-C\|\eta_{\varepsilon}\|_{L^2(\Omega)}^2.
\end{equation*}
Since $\|u_{\varepsilon}\|_{L^{2_s^*}(\Omega)}$ does not depends on $\varepsilon$ and $\displaystyle \widetilde{S}(\Sigma_{\mathcal{D}})=2^{\frac{-2s}{N}}S(N,s)$, by \eqref{eq:trunc_sobolev} and Lemma \ref{lemma_est_L2}, we have 
\begin{equation*}
\begin{split}
\|\eta_{\varepsilon}\|_{\mathcal{X}_{\Sigma_{\mathcal{D}}}^s(\mathscr{C}_{\Omega})}^2&\leq\widetilde{S}(\Sigma_{\mathcal{D}})+O(\varepsilon^{(1-\alpha)(N-2s)}),\\
\|\eta_{\varepsilon}\|_{L^2(\Omega)}^2&=\left\{
        \begin{tabular}{lr}
        $O(\varepsilon^{2s})$ & if $N>4s$, \\
        $O(\varepsilon^{2s}\log(1/\varepsilon))$  & if $N=4s$. \\
        \end{tabular}
        \right.
        \end{split}
\end{equation*}
Thus, for $N>4s$ we get 
\begin{equation*}
g_{\varepsilon}(t_{\varepsilon})\leq\frac{s}{N}[\widetilde{S}(\Sigma_{\mathcal{D}})]^{\frac{N}{2s}}+C_1\varepsilon^{(1-\alpha)(N-2s)}-C_2\varepsilon^{2s}<\frac{s}{N}[\widetilde{S}(\Sigma_{\mathcal{D}})]^{\frac{N}{2s}}=c_{\mathcal{D-N}}^*,
\end{equation*}
for $\varepsilon>0$ small enough and 
\begin{equation}\label{alpha_condition}
0<\alpha<\frac{N-4s}{N-2s}.
\end{equation}
Note that a similar relation between the concentration parameter $\varepsilon>0$ and the cut-off radius $\rho=\varepsilon^{\alpha}$ was obtained in \cite[Lemma 3.2]{Grossi1990}.\newline
If $N=4s$, the same conclusion follows. The case $2s<N<4s$ follows using the inequality \eqref{ineq:ab}, which gives, for some $\mu'>0$,
\begin{equation}\label{ineq:fin}
G_{\lambda}(w)\geq\frac{1}{2_s^*}w^{2_s^*}+\mu'w_0w^{2_s^*-1}.
\end{equation}
The result then follows arguing in a similar way as above, using \eqref{ineq:fin} together with the second estimate in Lemma \ref{lemma_est_L2}.
\end{proof}
\begin{remark}
Let us recall the value $c_{\mathcal{D}}^*=\frac{s}{N}[S(N,s)]^{\frac{N}{2s}}$ establishes the level below of which a local Palais--Smale condition holds for \eqref{p_lambda} under Dirichlet boundary condition (cf. \cite[Lemma 3.5]{Barrios2012}). Then, since $\displaystyle \widetilde{S}(\Sigma_{\mathcal{D}})\leq2^{\frac{-2s}{N}}S(N,s)$, we have
\begin{equation*}
c_{\mathcal{D}-\mathcal{N}}^*\leq\frac{1}{2}c_{\mathcal{D}}^*.
\end{equation*}
\end{remark}
\begin{proof}[Proof of Theorem \ref{sublinear}-(4)]
Let us fix $\lambda\in(0,\Lambda)$. In case of having $\displaystyle \widetilde{S}(\Sigma_{\mathcal{D}})<2^{\frac{-2s}{N}}S(N,s)$, since $\lim\limits_{t\to+\infty}\widetilde{J}_{\lambda}(t\tilde{w})=-\infty$, there exists $M\gg1$ such that $\widetilde{J}_{\lambda}(M\tilde{w})<\widetilde{J}_{\lambda}(0)$. On the other hand, if $\displaystyle \widetilde{S}(\Sigma_{\mathcal{D}})=2^{\frac{-2s}{N}}S(N,s)$, since $\lim\limits_{t\to+\infty}\widetilde{J}_{\lambda}(t\eta_{\varepsilon})=-\infty$, there exists $M_{\varepsilon}\gg1$ such that $\widetilde{J}_{\lambda}(M_{\varepsilon}\eta_{\varepsilon})<\widetilde{J}_{\lambda}(0)$. By Corollary \ref{cor:min_0_ext}, there exists $\rho>0$ such that, if $\|w\|_{\mathcal{X}_{\Sigma_{\mathcal{D}}}^s(\mathscr{C}_{\Omega})}=\rho$, then $\widetilde{J}_{\lambda}(w)\geq \widetilde{J}_{\lambda}(0)$. Next, if $\displaystyle \widetilde{S}(\Sigma_{\mathcal{D}})<2^{\frac{-2s}{N}}S(N,s)$, let us define
\begin{equation*}
\Gamma=\{\gamma\in\mathcal{C}([0,1],\mathcal{X}_{\Sigma_{\mathcal{D}}}^s(\mathscr{C}_{\Omega})):\ \gamma(0)=0,\gamma(1)=M\tilde{w}\},
\end{equation*}
and the minimax value
\begin{equation*}
c=\inf\limits_{\gamma\in\Gamma}\sup\limits_{0\leq t\leq1}\widetilde{J}_{\lambda}(\gamma(t)).
\end{equation*}
By the arguments above, $c\geq\widetilde{J}_{\lambda}(0)$. Moreover, by Lemma \ref{lem:below_level0}, we have 
\begin{equation*}
c\leq\sup\limits_{0\leq t\leq1}\widetilde{J}_{\lambda}(tM\tilde{w})=\sup\limits_{t\geq0}\widetilde{J}_{\lambda}(t\tilde{w})<c_{\mathcal{D-N}}^*.
\end{equation*}
On the other hand, if $\displaystyle \widetilde{S}(\Sigma_{\mathcal{D}})=2^{\frac{-2s}{N}}S(N,s)$, let us define
\begin{equation*}
\Gamma_{\varepsilon}=\{\gamma\in\mathcal{C}([0,1],\mathcal{X}_{\Sigma_{\mathcal{D}}}^s(\mathscr{C}_{\Omega})):\ \gamma(0)=0,\gamma(1)=M_{\varepsilon}\eta_{\varepsilon}\},
\end{equation*}
and the minimax value
\begin{equation*}
c_{\varepsilon}=\inf\limits_{\gamma\in\Gamma_{\varepsilon}}\sup\limits_{0\leq t\leq1}\widetilde{J}_{\lambda}(\gamma(t)).
\end{equation*}
By the arguments above, $c_{\varepsilon}\geq\widetilde{J}_{\lambda}(0)$. Moreover, by Lemma \ref{lem:below_level}, we have 
\begin{equation*}
c_{\varepsilon}\leq\sup\limits_{0\leq t\leq1}\widetilde{J}_{\lambda}(tM_{\varepsilon}\eta_{\varepsilon})=\sup\limits_{t\geq0}\widetilde{J}_{\lambda}(t\eta_{\varepsilon})<c_{\mathcal{D-N}}^*,
\end{equation*}
for $\varepsilon\ll1$ small enough. Thus, by Lemma \ref{lem:PS} and the Mountain Pass theorem (cf. \cite{Ambrosetti1973}) if $c>\widetilde{J}_{\lambda}(0)$ (resp. $c_{\varepsilon}>\widetilde{J}_{\lambda}(0)$), or the refinement of the MPT (cf. \cite{Ghoussoub1989}) if $c>\widetilde{J}_{\lambda}(0)$ (resp.$c_{\varepsilon}=\widetilde{J}_{\lambda}(0)$), we get the
existence of a non-trivial solution to $(\widetilde{P}_{\lambda})$, provided $u\equiv0$ is its unique solution. Of course this is a contradiction. Hence, $\widetilde{J}_{\lambda}$ admits a critical point $\tilde{w}>0$ so that $\widetilde{I}_{\lambda}$ admits a nontrivial critical point $\tilde{u}=\tilde{w}(x,0)>0$. As a consequence, $\hat{u}=u_0+\tilde{u}$ is a solution, different of $u_0$, to \eqref{p_lambda}. 
\end{proof}
%%%%%%%%%%%%%%%%%%%%%%%%%%%%%%%%%%%%
%%%%%%%%%%%%%%%%%%%%%%%%%%%%%%%%%%%%

\section{Convex case, $1<q<2_s^*-1$}\label{Section_convex}
In this section we address problem \eqref{p_lambda} in the convex setting $q>1$. Since the arguments carried out in the former section works with minor modifications in this convex case we will only indicate the main differences.
First, we have that the functional $J_{\lambda}$ has the appropriate geometry.
\begin{proposition}\label{prop:mpt_geom}
Let $\lambda>0$ and $1<q<2_s^*-1$. Then, the functional $J_{\lambda}$ has the Mountain Pass geometry. That is, there exists $\rho>0$ and $\beta>0$ such that 
\begin{enumerate}
\item $J_{\lambda}(0)=0$,
\item for all $w\in\mathcal{X}_{\Sigma_{\mathcal{D}}}^s(\mathscr{C}_{\Omega})$ with $\|w\|_{\mathcal{X}_{\Sigma_{\mathcal{D}}}^s(\mathscr{C}_{\Omega})}=\rho$ we have $J_{\lambda}(w)\geq\beta$,
\item there exists a positive function $h\in \mathcal{X}_{\Sigma_{\mathcal{D}}}^s(\mathscr{C}_{\Omega})$ such that $\|h\|_{\mathcal{X}_{\Sigma_{\mathcal{D}}}^s(\mathscr{C}_{\Omega})}>\rho$ and $J_{\lambda}(h)<\beta$.
\end{enumerate}
\end{proposition}
\begin{proof}
The proof follows as in \cite[Proposition 3.1]{Barrios2015}, so we omit the details.
\end{proof}
Since the proof of Lemma \ref{lem:PS} can be adapted for this convex case $q>1$, the main point in order to prove Theorem \ref{superlinear} is then to show that we can find a local PS$_c$ sequence with energy level under the critical level $c_{\mathcal{D-N}}^*$. This step follows the same scheme as in the concave case. Ii $\displaystyle \widetilde{S}(\Sigma_{\mathcal{D}})<2^{\frac{-2s}{N}}S(N,s)$, we use the extremal functions of the constant $\displaystyle \widetilde{S}(\Sigma_{\mathcal{D}})$. Otherwise, if $\displaystyle \widetilde{S}(\Sigma_{\mathcal{D}})=2^{\frac{-2s}{N}}S(N,s)$, we proceed as in Lemma \ref{lem:below_level} using now the estimate (cf. \cite[Lemma 3.4]{Barrios2015}),
\begin{equation}\label{estim}
\|\eta_{\varepsilon}\|_{L^{q+1}(\Omega)}^{q+1}\geq C\varepsilon^{N-\left(\frac{N-2s}{2}\right)(q+1)},\quad\text{for }N>2s\left(1+\frac{1}{q}\right),
\end{equation}
with $\eta_{\varepsilon}$ as in \eqref{eq:eta}. Note that in this case, there is no restriction on the size of the parameter $\lambda>0$. We conclude the proof by using the Mountain Pass Theorem.
\begin{proof}[Proof of Theorem \ref{superlinear}]
Assume $\displaystyle \widetilde{S}(\Sigma_{\mathcal{D}})=2^{\frac{-2s}{N}}S(N,s)$ (the other case is similar). Let us define 
\begin{equation*}
\Gamma_{\varepsilon}=\{\gamma\in\mathcal{C}([0,1],\mathcal{X}_{\Sigma_{\mathcal{D}}}^s(\mathscr{C}_{\Omega})):\ \gamma(0)=0,\gamma(1)=M_{\varepsilon}\eta_{\varepsilon}\},
\end{equation*}
for some $M_{\varepsilon}\gg1$ such that $J_{\lambda}(M_{\varepsilon}\eta_{\varepsilon})<0$ with $\eta_{\varepsilon}$ defined as in \eqref{eq:eta}. Note that, for any $\gamma\in\Gamma_{\varepsilon}$ the function $t\mapsto\|\gamma(t)\|_{\mathcal{X}_{\Sigma_{\mathcal{D}}}^s(\mathscr{C}_{\Omega})}$ is continuous in $[0,1]$. Then, since $\|\gamma(0)\|_{\mathcal{X}_{\Sigma_{\mathcal{D}}}^s(\mathscr{C}_{\Omega})}=0$ and $\|\gamma(1)\|_{\mathcal{X}_{\Sigma_{\mathcal{D}}}^s(\mathscr{C}_{\Omega})}=\|M_{\varepsilon}\eta_{\varepsilon}\|_{\mathcal{X}_{\Sigma_{\mathcal{D}}}^s(\mathscr{C}_{\Omega})}>\rho$ for $M_{\varepsilon}$ large enough, there exists $t_0\in(0,1)$ such that $\|\gamma(t_0)\|_{\mathcal{X}_{\Sigma_{\mathcal{D}}}^s(\mathscr{C}_{\Omega})}=\rho$ for $\rho$ given in Proposition \ref{prop:mpt_geom}. As a consequence,  
\begin{equation*}
\sup\limits_{0\leq t\leq1}J_{\lambda}(\gamma(t))\geq J_{\lambda}(\gamma(t_0))\geq\inf\limits_{\|g\|_{\mathcal{X}_{\Sigma_{\mathcal{D}}}^s(\mathscr{C}_{\Omega})}=\rho}J_{\lambda}(g)\geq\beta>0
\end{equation*}
with $\beta>0$ given in Proposition \ref{prop:mpt_geom}. Thus,
\begin{equation*}
c_{\varepsilon}=\inf\limits_{\gamma\in\Gamma_{\varepsilon}}\sup\limits_{0\leq t\leq1}J_{\lambda}(\gamma(t))>0.
\end{equation*}
By the Mountain Pass Theorem (cf. \cite{Ambrosetti1973}) we conclude that the functional $J_{\lambda}$ has a critical point $w\in\mathcal{X}_{\Sigma_{\mathcal{D}}}^s(\mathscr{C}_{\Omega})$ provided $N>2s\left(1+\frac{1}{q}\right)$. Moreover, since $J_{\lambda}(w)=c_{\varepsilon}\geq\beta>0$ and $J_{\lambda}(0)=0$ the function $w\neq0$. Therefore, $u=w(x,0)$ is a nontrivial solution to \eqref{p_lambda} for $q>1$ and $\lambda>0$.
\end{proof}

%%%%%%%%%%%%%%%%%%
%\bibliography{Biblio}{}
%\bibliographystyle{mystyle}

\end{document}